\documentclass{article}
\usepackage{multicol,graphicx,color}
\usepackage{pslatex}
\usepackage{authblk}
\usepackage{amsthm}
\usepackage{amsmath}
\usepackage{amssymb}
\usepackage{latexsym}
\usepackage{lscape}
\usepackage{epsfig}
\usepackage{pstricks}
\usepackage{amsfonts}
\usepackage{mathrsfs}
\usepackage{mathrsfs}
\UseRawInputEncoding
\usepackage[
  hmarginratio={1:1},     
  vmarginratio={1:1},     
  textwidth=15cm,        
  textheight=21cm,
  heightrounded,          
]{geometry}

\usepackage{graphicx,color}
\usepackage[colorlinks]{hyperref}
\hypersetup{linkcolor=blue,citecolor=blue,filecolor=black,urlcolor=blue}

\theoremstyle{plain}
\newtheorem{theorem}{Theorem}
\newtheorem{corollary}[theorem]{Corollary}
\newtheorem{lemma}[theorem]{Lemma}
\newtheorem{proposition}[theorem]{Proposition}
\theoremstyle{definition}

\numberwithin{equation}{section}
\numberwithin{theorem}{section}

\allowdisplaybreaks[4]

\author{Xiaojun Chang \footnote{changxj100@nenu.edu.cn}} \affil{School of Mathematics and Statistics \& Center for Mathematics and Interdisciplinary Sciences,
 Northeast Normal University, Changchun 130024, Jilin,
PR China}

\author{Manting Liu \footnote{liumt679@nenu.edu.cn}} \affil{School of Mathematics and Statistics,
 Northeast Normal University, Changchun 130024, Jilin, PR China}

\author{Duokui Yan \footnote{duokuiyan@buaa.edu.cn}} \affil{School of Mathematical Sciences,
Beihang University, Beijing 100191, PR China }

\title{Positive normalized solutions of Schr\"{o}dinger equations with Sobolev critical growth in bounded domains}
\date{}

\begin{document}

\maketitle

\begin{abstract}
\noindent This paper investigates the existence of positive normalized solutions to the Sobolev critical Schr\"{o}dinger equation:
\begin{equation*}
\left\{
\begin{aligned}
&-\Delta u +\lambda u =|u|^{2^*-2}u \quad &\mbox{in}& \ \Omega,\\
&\int_{\Omega}|u|^{2}dx=c, \quad  u=0 \quad &\mbox{on}& \ \partial\Omega,
\end{aligned}
\right.
\end{equation*}
where $\Omega\subset\mathbb{R}^{N}$ ($N\geq3$) is a bounded smooth domain, $2^*=\frac{2N}{N-2}$, $\lambda\in \mathbb{R}$ is a Lagrange multiplier, and $c>0$ is a prescribed constant. By introducing a novel blow-up analysis for Sobolev subcritical approximation solutions with uniformly bounded Morse index and fixed mass, we establish the existence of mountain pass type positive normalized solutions for
 $N\ge 3$. This resolves an open problem posed in  [Pierotti, Verzini and Yu, SIAM J. Math. Anal. 2025].
\end{abstract}

\medskip

{\small \noindent \text{Key Words:}  Normalized solutions; Sobolev critical growth;
Bounded domains; Blow-up analysis; Variational methods.}

\medskip

\section{Introduction and main results}\label{intro}

In this paper, we investigate the following nonlinear Schr\"{o}dinger equation
\begin{equation}\label{1.1-Lapla}
\left\{
\begin{aligned}
&-\Delta u+ \lambda u=|u|^{2^*-2}u\quad  &\mbox{in}& \ \Omega,\\
& \int_{\Omega}|u|^{2}dx=c, \ u=0 \quad &\mbox{on}& \ \partial\Omega,
\end{aligned}
\right.
\end{equation}
where $\Omega\subset\mathbb{R}^{N}(N\geq3)$ is a bounded smooth domain, $c>0$ is a prescribed constant, and $\lambda\in\mathbb{R}$ is an unknown parameter that acts as a Lagrange multiplier. A function $u\in H_0^1(\Omega)$ that satisfies (\ref{1.1-Lapla}) for some $\lambda\in \mathbb{R}$ is referred to as a normalized solution. These solutions are of particular interest due to their significant physical relevance in various fields, including nonlinear optics and Bose-Einstein condensation. In these applications, the $L^2$-norm of the solution remains invariant under the evolution,  which is of fundamental importance as it reflects the principle of mass conservation.  Moreover, the variational characterization of normalized solutions often provides a powerful tool for analyzing their orbital stability and instability properties. For further details on bounded domains, we refer to \cite{Ag2013,BP2014,FM2001,NTV2015,NTV2019}. The case $\Omega=\mathbb{R}^N$ is addressed in \cite{BC2013,Caze82,JJLV2022,S1,S2}. In addition, normalized solutions also exhibit a deep connection with Mean Field Games systems, see \cite{CCV2024, PPVV2021}.

We introduce the energy functional $J:H^{1}_{0}(\Omega)\to \mathbb{R}$ as follows:
\begin{equation*}
J(u)
:=\frac{1}{2}\int_{\Omega}|\nabla u|^2 dx-\frac{1}{2^*}\int_{\Omega}|u|^{2^*}dx.
\end{equation*}
Then $J$ is of $C^1$, and the critical points of $J$ constrained to the $L^2$-sphere
$$\mathcal{S}_{c}:=\left\{u\in H^{1}_{0}(\Omega): \int_{\Omega}|u|^{2}dx=c \right\}$$
correspond to normalized solutions of (\ref{1.1-Lapla}).

Significant progress has been made in the study of normalized solutions when $\Omega =\mathbb{R}^N$, following Jeanjean's breakthrough work \cite{J}.
 In that work, he pioneered the use of dilations and a mountain pass argument applied to a scaled functional to explore normalized solutions for nonlinear Schr\"odinger equations (NLS) exhibiting $L^2$-supercritical growth. A key innovation in his proof was the identification of a special Palais-Smale (PS) sequence closely linked to the Pohozaev identity. Further developments by Bartsch and Soave \cite{BS1} introduced a natural constraint approach to investigate normalized solutions, showing that
the intersection of $\mathcal{S}_c$ with the Pohozaev set serves as a natural constraint. Building upon these insights, various
variational methods utilizing the Pohozaev identity have been developed to establish the existence of normalized solutions, as detailed in \cite{BM,CLY2023,IT2019,JL,S1} and related references.

In \cite{S2}, Soave employed the Pohozaev manifold decomposition to explore normalized solutions to the NLS in $\mathbb{R}^N$ with mixed nonlinearities of the form $\mu |u|^{q-2}u+|u|^{2^*-2}u$, where $\mu>0$ and $q\in (2, 2^*)$. Specifically, \cite{S2} established the existence of local minimizers for $q\in (2, 2+\frac{4}{N})$ and derived mountain pass type critical points for $q\in [2+\frac{4}{N}, 2^*)$. This work provides a normalized solutions counterpart to the classical Brezis-Nirenberg problem in $\mathbb{R}^N$, with the added challenge of addressing Sobolev critical nonlinearity under an $L^2$-constraint. 

Note that for the $L^2$-subcritical range $q\in (2, 2+\frac{4}{N})$, the geometry of the associated energy functional suggests the existence of a second positive solution at the mountain pass level, a conjecture formally raised by Soave in \cite{S2}. To prove this assertion, Jeanjean and Le \cite{JL2022} developed a non-radial energy estimate to establish a mountain pass type positive normalized solution for $N\ge 4$. Meanwhile,  Wei and Wu \cite{WW2022} employed radial test functions inspired by the Aubin-Talenti extremals to derive new estimates for
$N\ge3$. For additional contributions to this line, we refer the reader to 
\cite{AJM2022,CT2024,JJLV2022,L2021,MS2022} and their references.

The study of normalized solutions within bounded domains is less explored and presents distinct challenges compared to the whole space $\mathbb{R}^N$.  While tools such as dilations and the Pohozaev identity are crucial in the unbounded setting, their application is significantly altered and often unavailable in bounded domains. 
In \cite{NTV2019}, Noris, Tavares, and Verzini utilized Ekeland's variational principle to establish the existence of local minimizers for small mass in the Sobolev critical case. In particular, they proved the following result:
\begin{theorem}[\cite{NTV2019}]\label{Thm-loc-mini}
For $N\geq 3$, there exists $c^{*}>0$ such that for any $0<c<c^{*},$ the problem (\ref{1.1-Lapla}) admits a positive solution that is a local minimizer.
\end{theorem}

Our objective is to investigate the existence of mountain pass type positive normalized solutions of problem (\ref{1.1-Lapla}). However, pursuing such solutions in bounded domains faces inherent obstacles: the failure of dilation techniques, the non-compactness of $H_{0}^{1}(\Omega)\hookrightarrow L^{2^{*}}(\Omega)$, and the complications arising from boundary terms in the Pohozaev identity. 
For star-shaped domains, recent work by Pierotti, Verzini and Yu \cite{PVY}, as well as Song and Zou \cite{SZ2024}, has circumvented these issues by combining Pohozaev-based analysis of bounded (PS) sequences with small mass regimes to identify a second positive normalized solution.

For general bounded domains, Pierotti et al.\cite{PVY} employed the Struwe-Jeanjean monotonicity \cite{J1999, Str1-1988, Str2-1988}, with mass $c$ as a parameter, to construct mountain pass solutions for $N\ge 3$, where $c$ lies in a positive measure subset of the interval $(0, \hat{c}^{*})$ for some threshold $\hat{c}^{*}>0$. Drawing from an idea in \cite{Str3-1988}, they established sharp estimates on the derivative of the mountain pass level. This estimate ensured the $H_0^1$-boundedness and strong convergence of the (PS) sequence $\{u_n\}$. Moreover, a recent result on sign-changing normalized solutions for (\ref{1.1-Lapla}) in general bounded domains was obtained using a descending flow invariant set method with a suitable linking structure. For related work on normalized solutions of NLS in bounded domains within the Sobolev subcritical regime, see \cite{CRZ2025, NTV, PPVV2021, PV} and references therein. 

In Remark 1.8 of \cite{PVY}, it was noted as an open problem that problem (\ref{1.1-Lapla}) in general bounded domains admits a second positive solution at the mountain pass level for $N\ge 3$, provided the mass $c$ lies within the entire small mass interval. The goal of this paper is to resolve this problem.
Our main result is the following.

\begin{theorem}\label{crit-theo-Lapla}
For $N\geq 3$, there exists $c^{**}>0$ such that for any $0<c<c^{**},$ problem (\ref{1.1-Lapla}) admits a second positive normalized solution, which is of mountain pass type.
\end{theorem}


Our strategy combines a Sobolev subcritical approximation scheme with a novel blow-up analysis tailored to bounded domains. We begin by considering a sequence of approximation solutions with uniformly bounded Morse indices on the $L^2$-sphere for the Sobolev subcritical problems. This sequence is derived using an abstract parameterized minimax principle from \cite{BCJS1}, which incorporates Morse index information for constraint functionals, along with a corresponding blow-up analysis for the solution sequence of the approximation problem, as discussed in \cite{EP,PVY}. This methodology has recently been successfully applied to study the existence of mountain pass type normalized solutions for $L^2$-supercritical but Sobolev subcritical NLS or Kirchhoff equations in bounded domains \cite{BQZ,CRZ2025,PVY, WC2024}, and for $L^2$-supercritical NLS on metric graphs \cite{BCJS,CJS2024}. However, it cannot be directly extended to the Sobolev critical problem, both in terms of recovering compactness for the bounded (PS) sequence of the associated parameterized problem and performing blow-up analysis for the solution sequence of the approximation problem.

To address the Sobolev critical issue, we instead focus on the sequence of approximation solutions for the Sobolev subcritical problems. The primary challenge lies in the blow-up analysis of this sequence. Traditional blow-up analyses fail in this case due to the energy concentration phenomena arising from the critical exponent $2^*$. To overcome this, we introduce a new blow-up function to analyze the subcritical approximation solutions, thereby facilitating the establishment of uniform boundedness. Furthermore, we employ a refined profile decomposition framework from \cite{GLW, T2007} to prove strong convergence in $H_0^1(\Omega)$. By comparing the mountain pass energy with that of the local minimizer from Theorem \ref{Thm-loc-mini}, we show that the limiting function constitutes the second positive normalized solution of mountain pass type.

More precisely, for any $p\in (2, 2^*)$, we consider the energy functional $J_p:H^{1}_{0}(\Omega)\to \mathbb{R}$ defined by
\begin{equation*}
J_p(u)
:=\frac{1}{2}\int_{\Omega}|\nabla u|^2 dx-\frac{1}{p}\int_{\Omega}|u|^{p}dx.
\end{equation*}
 According to Proposition \ref{subcrit-theo-Lapla}, there exists a sequence $\{u_{p_n}\}\subset\mathcal{S}_{c}$  with $p_n\to 2^*$ satisfying (\ref{1.1-Lapla})
for some $\{\lambda_n\}\subset\mathbb{R}$.  A notable feature of $\{u_{p_n}\}$ is that it possesses a uniformly bounded Morse index, i.e., $m(u_{p_n})\le 2$, and a fixed mass, i.e., $\|u_{p_n}\|_2^2=c$. A direct computation yields
\begin{equation*}
\|\nabla u_{p_{n}}\|_{2}^{2}=\frac{2p_{n}}{p_{n}-2}\left(c_{1,p_{n}}+\frac{\lambda_{n}}{p_{n}}c\right),
\end{equation*}
where $c_{1,p_n}=J_{p_n}(u_{p_{n}})$.
Since $\Omega$ is a general bounded domain, we cannot establish an upper bound estimate for $\{c_{1,p_{n}}\}$ by means of the Pohozaev identity as in \cite{L2021}. Nevertheless, by restricting the range of the mass $c$, we observe that as $p_n\to 2^*$, the functionals $J_{p_n}$ and $J$ share a common mountain pass path set. This allows us to apply the Lebesgue dominated convergence theorem to derive the desired upper bound estimate, as discussed in Section 2.
Consequently, to prove that $\{u_{p_{n}}\}$ is bounded in $H_0^1(\Omega)$, it suffices to find an upper bound for $\{\lambda_n\}$.

The difficulty in bounding $\{\lambda_n\}$ through blow-up analysis arises from obtaining uniform decay estimates for solution sequences away from their local maxima. While comparison principles (\cite{BCJS, CJS2024, PV, WC2024}) generally suffice for subcritical cases, they fail in the Sobolev critical regime. To overcome this, inspired by \cite{P2013}, we introduce a specialized blow-up function $\Theta_n^k(x)$ designed for the Sobolev critical exponent.  Unlike classical blow-up techniques \cite{EP,NTV}, our scaled sequences do not achieve their maxima at the origin but maintain uniform boundedness in arbitrary domains, with the origin emerging as a local maximum in the limiting functions. 
By careful selecting test functions and their supports, we establish uniform Morse index bounds for the limit functions, thereby deriving global decay estimates for $\Theta_n^k(x)$. These estimates eventually provide the desired upper bound for $\{\lambda_n\}$, see Section 3.

Once the boundedness of $\{u_n\}$ is ascertained, we proceed to show that $\{u_{p_n}\}$ strongly converges to some $\hat u\in H_0^1(\Omega)$. To restore compactness in the Sobolev critical setting, we refine the profile decomposition from \cite{GLW} and combine it with our uniform boundedness result. This avoids relying on local Pohozaev identities and dimensional constraints as in \cite{GLW}, allowing us instead to use direct estimates on the term $\int_{D_n}u_{p_n}^2\phi_n dx$ (with $D_n$ and $\phi_n$ defined in Section 3) to derive a contradiction. This ultimately leads to the existence of positive solutions to (\ref{1.1-Lapla}), which is distinguished from the local minimizer of the functional $J$ through energy comparison. The energy gap then validates the existence of mountain pass type positive normalized solutions.

The paper is organized as follows. Section 2 establishes some preliminary results, focusing particularly on energy estimates for Sobolev subcritical approximation solutions and the local minimum solution. Section 3 offers a blow-up analysis and demonstrates that the Lagrangian multiplier sequence is bounded. In Section 4, we provide the proof of Theorem \ref{crit-theo-Lapla}.

Throughout this paper, we adopt the following notations.
\begin{itemize}
  \item The symbols $C,$ $C_{i}$ denote various positive constants that may vary from line to line.
    \item For $N\in \mathbb{N}$, we define $\mathbb{R}^{N}_{+}:=\{x=(x_{1},\cdots, x_{N})\in \mathbb{R}^N: x_{N}>0\}$ and  $\mathbb{R}^{N}_{-}:=\mathbb{R}^{N}\setminus\mathbb{R}^{N}_{+}$.
  \item For any $r>0$ and $z\in\mathbb{R}^{N},$ $B_{r}(z)$ denotes the open ball of radius $r$ centered at $z$. For simplicity, we write $B_{r}:= B_{r}(0)$. 
   \item For any $r>0$ and $z\in\mathbb{R}^{N},$  we define $B_r^+(z):=B_r(z)\cap \mathbb{R}^N_+$. For simplicity, we write $B_{r}^+:= B_{r}^{+}(0)$.

  \item For $1\leq p\leq \infty$, the $L^p$-norm of a function $u\in L^{p}(\Omega)$ is denoted by $\|u\|_{p}.$   
\end{itemize}

\section{Preliminary}

In this section, we present several preliminary results.  We begin by recalling the Gagliardo-Nirenberg inequality in bounded domains (see \cite{N-1966}):
For every $N\geq 3$ and $p\in(2,2^*)$, there exists a constant $C_{N,r,\Omega}$ depending on $N$, $r$ and $\Omega$ such that
\begin{eqnarray*}
\|u\|_r\leq C_{N,r,\Omega}\|u\|_2^{1-\gamma_r}\|\nabla u\|_{2}^{\gamma_r},~~\forall u\in H_0^1(\Omega),
\end{eqnarray*}
where $\gamma_r:=\frac{N(r-2)}{2r}$. Moreover, as shown in \cite{NTV2019}, this inequality holds for any bounded domain $\Omega$ with the same constant $C_{N,p}$, where $C_{N,p}$ denotes the sharp constant in the Galiardo-Nirenberg inequality within $\mathbb{R}^N$.

We now proceed with the following result concerning Sobolev subcritical problems.

\begin{proposition}\label{subcrit-theo-Lapla}
Assume that $p\in(2+\frac{4}{N},2^{*})$ and $\Omega\subset\mathbb{R}^N(N\ge3)$ is a bounded smooth domain.Then there exists a constant  $c^*_p>0$ such that for any $c\in (0, c^*_p)$, the following problem
\begin{equation}\label{4.1-Lapla}
\left\{
\begin{aligned}
 & -\Delta u+\lambda u=u^{p-1} \quad &\mbox{in}& \ \Omega,\\
 & \int_{\Omega}|u|^{2}dx=c,\ \  u=0 \quad &\mbox{on}& \ \partial\Omega
\end{aligned}
\right.
\end{equation}
has a positive normalized solution $u\in H_0^1(\Omega)$ for some $\lambda\in \mathbb{R}$, which is of mountain pass type.
Moreover, the Morse index of $u$, denoted by $m(u)$, satisfies $m(u)\leq 2,$ where $m(u)$ is defined as
\begin{equation*}
m(u):=\sup\{dim W: W\subset H_{0}^{1}(\Omega),\forall \phi\in W\setminus \{0\},Q_{\lambda,u}(\phi)<0\},
\end{equation*}
and the quadratic form $Q_{\lambda,u}$ is given by
\begin{equation*}
Q_{\lambda,u}(\phi):=\int_{\Omega}|\nabla \phi|^{2}dx+\int_{\Omega}(\lambda-(p-1)u^{p-2})\phi^{2}dx.
\end{equation*}
\end{proposition}

The result described above is essentially derived through the monotonicity arguments in \cite{CJS2024} and the blow-up analysis developed in \cite{EP}. See also \cite{PV,PVY}.
For the sake of clarity and to facilitate future applications in proving Theorem \ref{crit-theo-Lapla}, we provide a brief outline of the proof below.

Fix $p\in (2+\frac{4}{N}, 2^*)$.
Consider the parameterized functional $J_{\rho,p}: H_0^1(\Omega)\to \mathbb{R}$ defined by
\begin{equation*}
J_{\rho,p}(u):=\frac{1}{2}\|\nabla u\|^{2}_{2}-\frac{\rho}{p}\int_{\Omega}|u|^{p}dx, \quad  \forall \rho\in\left[\frac{1}{2},1\right], u\in H_0^1(\Omega).
\end{equation*}
We first show that $J_{\rho,p}$ exhibits a mountain pass geometry on $\mathcal{S}_{c}$ uniformly with respect to $\rho\in[\frac{1}{2},1]$. 
Set 
\begin{eqnarray*}\label{m-1}
c^*_p:=2^{-N}\bigg(\frac{p}{C_{N,p}^p}\bigg)^{\frac{2}{p-2}}\left(\frac{1}{2\lambda_1(\Omega)}\right)^{\frac{N}{2}-\frac{2}{p-2}},
\end{eqnarray*}
where $\lambda_{1}(\Omega)$ denotes the first Dirichlet eigenvalue of $(-\Delta,H^{1}_{0}(\Omega))$.

\begin{lemma}\label{MP-La-1} For any $c\in(0,c^{*}_{p}),$ there exist $w_{1}, w_{2}\in \mathcal{S}_{c}$, independent of $\rho$ and $p$, such that
\begin{equation*}
c_{\rho,p}:=\inf_{\gamma\in\Gamma}\max_{t\in[0,1]}J_{\rho,p}\big(\gamma(t)\big)>\max\left\{J_{\rho,p}(w_{1}), J_{\rho,p}(w_{2})\right\}, \quad \forall \rho\in\left[\frac{1}{2},1\right],
\end{equation*}
where
\begin{equation*}
\Gamma:=\left\{\gamma\in C([0,1],\mathcal{S}_{c}) :\gamma(0)=w_{1}, \gamma(1)=w_{2}\right\}.
\end{equation*}
\end{lemma}
\begin{proof}
Let $\varphi_1$ be the corresponding normalized eigenfunction associated with $\lambda_1(\Omega)$ such that $\|\varphi_1\|_{2}^2=1$. For $\alpha> \lambda_{1}(\Omega),$ we define the sets
\begin{equation*}\label{def-set}
  \mathcal{D}_{\alpha}:=\left\{u\in\mathcal{S}_{c}:\|\nabla u\|_{2}^{2}\leq c \alpha\right\}
\quad \mbox{and} \quad
  \mathcal{B}_{\alpha}:=\left\{u\in\mathcal{S}_{c}:\|\nabla u\|_{2}^{2}=c \alpha\right\}.
\end{equation*}
It is easily seen that $\sqrt{c}\varphi_1\in  \mathcal{D}_{\alpha}$. For $\alpha>0$, we introduce the function
\[
\phi_\alpha(x):=\sqrt{c}\left(\frac{\alpha}{\lambda_1(\Omega)}\right)^{\frac{N}{4}}\tilde{\varphi}_1\left(\sqrt{\frac{\alpha}{\lambda_1(\Omega)}}(x-x_0)\right), \forall x\in \Omega,
\]
where $x_0\in\Omega$, $\tilde{\varphi}_1$ is the zero extension of $\varphi_1$ to $\mathbb{R}^N$.
Clearly, there exists $\bar{\alpha}>\lambda_1(\Omega)$ such that $\phi_{\alpha}\in  \mathcal{B}_{\alpha}$ for all $\alpha\geq\bar{\alpha}$.

Let $w_1=\phi_{2\lambda_1(\Omega)}$. 
Then, for any $\rho\in [\frac{1}{2},1]$ and $c\in(0,c_p^*)$, we deduce
\begin{eqnarray}\label{9-24-1}
J_{\rho,p}(w_1)
\leq\frac{1}{2}\int_{\Omega}\vert\nabla w_1\vert^2dx<2c\lambda_1(\Omega)\le  \inf_{u\in B_{8\lambda_1(\Omega)}}J_{\rho,p}(u).
\end{eqnarray}
Similarly, since $p\in(2+\frac{4}{N},2^{*})$, we can choose sufficiently large $\alpha_0>0$ such that, letting $w_2=\phi_{\alpha_0},$ we have
\begin{equation}\label{9-24-2}
\|\nabla w_2\|_{2}^{2}>16c \lambda_{1}(\Omega) \quad  \mbox{and} \quad  J_{\rho,p}(w_2) \leq J_{\frac{1}{2},p}(w_2)<0, \quad \forall \rho\in\left[\frac{1}{2},1\right].
\end{equation}
By continuity, for any $\gamma\in\Gamma,$ there exists $t_{\gamma}\in[0,1]$ such that $\gamma(t_{\gamma})\in\mathcal{B}_{8\lambda_1(\Omega)},$ which implies
\begin{equation*}
\begin{aligned}
\max\limits_{t\in[0,1]}J_{\rho,p}(\gamma(t))\geq J_{\rho,p}(\gamma(t_{\gamma}))
\geq\inf_{u\in \mathcal{B}_{8\lambda_1(\Omega)}}J_{\rho,p}(u)\ge 2c\lambda_1(\Omega).
\end{aligned}
\end{equation*}
Since $\gamma\in\Gamma$ is arbitrary, by (\ref{9-24-1})-(\ref{9-24-2}), we conclude that
\begin{equation*}
c_{\rho,p}\geq
2c\lambda_1(\Omega) >\max\{J_{\rho,p}(w_1),J_{\rho,p}(w_2)\}.
\end{equation*}
\end{proof}

\begin{proof}[Proof of Proposition \ref{subcrit-theo-Lapla}]
By Lemma \ref{MP-La-1} and applying Theorem 1.5 from \cite{BCJS1}, for almost every $\rho \in \left[\frac{1}{2}, 1\right]$, we obtain a bounded nonnegative Palais-Smale sequence $\{u_{n,\rho}\}\subset H_0^1(\Omega)\times \mathbb{R}$ for $J_{\rho,p}$ constrained on $\mathcal{S}_{c}.$ Since the embedding $H^{1}_{0}(\Omega)\hookrightarrow L^{r}(\Omega)$ is compact for $r\in[1,2^*)$, it follows that $u_{n,\rho}\to u_{\rho}$ in $H^{1}_{0}(\Omega)$ as $n\to+\infty$.
By the strong maximum principle, we conclude that $u_{\rho}>0.$
Moreover, following similar arguments as in \cite{CJS2024} (see also \cite{CRZ2025}), we have $m(u_{\rho})\leq 2.$

Next, we consider a sequence $\rho_n \rightarrow 1^-$ and the corresponding sequence $\{u_{\rho_n}\}\subset H_0^1(\Omega)$ of mountain pass type critical points of $J_{\rho_n,p}$ on $\mathcal{S}_c$ at the level $c_{\rho_n, p}$, with Morse index $m(u_{\rho_n}) \leq 2$. Let $\{\lambda_n\}\subset \mathbb{R}$ denote the associated Lagrange multipliers.
Employing the blow-up analysis as in \cite{EP,PV}, we infer that $\{(u_{\rho_n}, \lambda_n)\}$ is bounded in $H_0^1(\Omega)\times \mathbb{R}$. Thus, by similar arguments as above, there exists $(u_0, \lambda_0)\in H_0^1(\Omega)\times \mathbb{R}$ that is a mountain pass type positive normalized solution to (\ref{4.1-Lapla}).
\end{proof}


Set
\begin{equation*}
c^{*}_{2^*}:=2^{-N}(2^*S^{\frac{2^*}{2}})^{\frac{2}{2^*-2}}\frac{1}{2\lambda_1(\Omega)},
\end{equation*}
where $S$ denotes the sharp constant in the Sobolev inequality within $\mathbb{R}^N$. 
By similar arguments as in Lemma \ref{MP-La-1}, we have the following result.
\begin{lemma}\label{MP-Lapla2} For any $c\in(0,c^{*}_{2^*}),$ we have
\begin{equation*}
c_{1,2^*}:=\inf_{\gamma\in\Gamma}\max_{t\in[0,1]}J\big(\gamma(t)\big)>\max\left\{J(w_{1}), J(w_{2})\right\},
\end{equation*}
where $w_{1}$ and $w_{2}$ are defined in Lemma \ref{MP-La-1}.
\end{lemma}

Now, let $p_{n}\in (2+\frac{4}{N},2^{*})$ be a sequence such that $p_{n}\to 2^{*}$ as $n\to+\infty$. According to \cite{B2002}, it follows that the sequence $\{c_{p_n}^*\}$ is bounded both above and below, with a positive lower bound. Consequently, for any $c\in(0,\inf c_{p_n}^*)$, by Proposition \ref{subcrit-theo-Lapla}, there exists a sequence of solution pair $\{(u_{p_{n}}, \lambda_n)\}\subset \mathcal{S}_{c}\times \mathbb{R}$ with $u_{p_n}>0$ such that $m(u_{p_n})\le 2$, and these pairs solve the equation:
\begin{equation}\label{C1-Lapla}
  -\Delta u_{p_{n}}+\lambda_{n} u_{p_{n}}= u_{p_{n}}^{p_{n}-1}, \quad x\in\Omega.
\end{equation}
Furthermore, we have
\begin{equation*}\label{2.1}
  c_{1,p_n}=\inf\limits_{\gamma\in\Gamma}\max\limits_{t\in[0,1]}J_{p_n}(\gamma(t))\geq 2c\lambda_1(\Omega) 
\end{equation*}
and
\begin{equation}\label{C6-Lapla}
\|\nabla u_{p_{n}}\|_{2}^{2}=\frac{2p_{n}}{p_{n}-2}\left(c_{1,p_{n}}+\frac{\lambda_{n}c}{p_{n}}\right).
\end{equation}


Set $c^{**}:=\min\{\inf c_{p_n}^*, c^{*}_{2^*}\}$.
We can show that the sequence $\{c_{1,p_n}\}$ is bounded above by $c_{1,2^{*}}$. 

\begin{lemma}\label{AB-Lapla}
For any $c\in (0, c^{**})$, we have $\limsup\limits_{n\to+\infty}c_{1,p_{n}}\leq c_{1,2^{*}}$.
\end{lemma}
\begin{proof}
For any $\epsilon>0$, by the definition of $c_{1,2^{*}}$, there exist $\gamma_{\epsilon}\in\Gamma$ and $t_{2^{*}}\in[0,1]$ such that
\begin{equation*}\label{14-1}
c_{1,2^{*}}\geq \max\limits_{t\in[0,1]}J(\gamma_{\epsilon}(t))-\epsilon=J(\gamma_{\epsilon}(t_{2^{*}}))-\epsilon.
\end{equation*}
By applying the Young's inequality, we obtain
\begin{equation*}
|\gamma_{\epsilon}(t)|^{p_{n}}
\leq \frac{2^{*}-p_{n}}{2^{*}-2}|\gamma_{\epsilon}(t)|^{2}+\frac{p_{n}-2}{2^{*}-2}|\gamma_{\epsilon}(t)|^{2^{*}}
\leq |\gamma_{\epsilon}(t)|^{2}+|\gamma_{\epsilon}(t)|^{2^{*}}, \quad \forall t\in[0,1].
\end{equation*}
Next, using the Lebesgue dominated convergence theorem, we deduce that the function
\begin{equation*}
  \frac{1}{p}\int_{\Omega}|\gamma_{\epsilon}(t)|^{p}dx
\end{equation*}
is continuous with respect to $p\in(2+\frac{4}{N},2^{*}]$ uniformly for $t\in[0,1]$. Since $p_{n}\to 2^{*}$,  we can choose $n$ sufficiently large such that
\begin{eqnarray*}\label{8-7-1}
|J_{p_{n}}(\gamma_{\epsilon}(t))-J(\gamma_{\epsilon}(t))|
=\left|\frac{1}{2^{*}}\int_{\Omega}|\gamma_{\epsilon}(t)|^{2^{*}}dx-\frac{1}{p_{n}}\int_{\Omega}|\gamma_{\epsilon}(t)|^{p_{n}}dx\right|<\epsilon, \quad \forall t\in[0,1],
\end{eqnarray*}
which implies that
\begin{equation*}
c_{1,p_{n}}\le \max\limits_{t\in[0,1]}J_{p_{n}}(\gamma_{\epsilon}(t))
\le \max\limits_{t\in[0,1]}J(\gamma_{\epsilon}(t))+\epsilon
\le c_{1,2^{*}}+2\epsilon.
\end{equation*}
Since $\epsilon>0$ is arbitrary, we get that the conclusion holds.
\end{proof}
Together with (\ref{C6-Lapla}) and Lemma \ref{AB-Lapla}, we deduce that
\begin{corollary}\label{lower bound}
There exists a constant $C_0>0$ independent of $n$ such that $\lambda_n\ge -C_0$ for all $n$.
\end{corollary}

To prove Theorem  \ref{crit-theo-Lapla}, it is vital to distinguish between the local minimizer of $J$ and the positive solution obtained via the subcritical approximation procedure. This distinction will be achieved by comparing the energy levels of these two solutions. The following result provides an upper bound on the energy of the local minimizer. 
\begin{lemma}\label{Mini}
For any $0<c<c_{2^*}^{*}$, the functional $J$ admits a local minimizer $u_{1}^{*}$ on $\mathcal{S}_{c}$ such that
\begin{equation}\label{D30}
J(u_{1}^{*})\leq c\lambda_1(\Omega).
\end{equation}
\end{lemma}
\begin{proof}
Set
$$m_{\mathcal{D}}:=\inf\limits_{u\in \mathcal{D}_{8\lambda_1(\Omega)}}J(u),~~~~
m_{\mathcal{B}}:=\inf\limits_{u\in \mathcal{B}_{8\lambda_1(\Omega)} }J(u).$$
We claim that, for $c\in(0,c_{2^*}^{*})$, 
\begin{equation}\label{D28}
m_{\mathcal{D}}<m_{\mathcal{B}}.
\end{equation}
In fact, noting $\mathcal{B}_{2\lambda_1(\Omega)}\subset\mathcal{D}_{8\lambda_1(\Omega)}$, we directly infer that
$m_{\mathcal{D}}\leq c\lambda_1(\Omega)$.
For any $u\in \mathcal{B}_{8\lambda_1(\Omega)},$ by the Sobolev inequality and the definition of $c_{2^*}^{*}$, we obtain
\begin{align*}
  J(u)
\geq 4c\lambda_1(\Omega)-\frac{1}{2^*}S^{-\frac{2^*}{2}}(8c\lambda_1(\Omega))^{\frac{^{2^{*}}}{2}}>2c\lambda_1(\Omega).
\end{align*}
Since $u\in \mathcal{B}_{8\lambda_1(\Omega)}$ is arbitrary, we get
$
m_{\mathcal{B}}\geq 2c\lambda_1(\Omega).
$
Hence, (\ref{D28}) holds. Therefore, there exists a bounded nonnegative minimizing sequence $\{u_{n}^{*}\}\subset \mathcal{D}_{8\lambda_1(\Omega)}\setminus\mathcal{B}_{8\lambda_1(\Omega)}$ such that $\lim\limits_{n\to\infty}J(u_{n}^{*})=m_{\mathcal{D}}$. By standard arguments as in \cite{NTV2019}, it follows that $u_{n}^{*}\to u_{1}^{*}$ in $H_{0}^1(\Omega)$ as $n\to+\infty.$ 
Then, we deduce that
$J$ admits a local minimizer $u_1^*\in \mathcal{D}_{8\lambda_1(\Omega)}\setminus\mathcal{B}_{8\lambda_1(\Omega)}$
and $(u_{1}^*,\lambda_{u_1^*})$ solves (\ref{1.1-Lapla}) for some $\lambda_{u_1^*}\in\mathbb{R}$. Furthermore, $J(u_{1}^{*})=m_{\mathcal{D}}\leq c\lambda_1(\Omega)$.
This completes the proof.
\end{proof}

\section{Blow-up analysis}
In this section, we conduct a blow-up analysis for the sequence 
 $\{u_{p_{n}}\}$.  We aim to establish the following theorem:
  \begin{theorem}\label{above-zero}
The sequence $\{\lambda_n\}$ is bounded from above.
\end{theorem}

To prove this, we assume by contradiction that
\begin{equation*}\label{lambda}
 \lambda_{n}\to+\infty~~\mbox{as}~~ n\to+\infty.
\end{equation*}
From (\ref{C1-Lapla}), we have  $\lambda_{n} \|u_{p_{n}}\|_{\infty}^{-(p_{n}-2)}\leq 1$. Then $\|u_{p_{n}}\|_{\infty}\to +\infty$ as $n\to+\infty$.

To derive a contradiction, we introduce a blow-up function $\Theta_n^k(x)$ and show that it satisfies a specific global decay estimate, which plays a pivotal role in our proof. 
\begin{theorem}\label{L-criti-blow1}
Let $\hat{\epsilon}_{n}:=\left(\|u_{p_{n}}\|_{\infty}\right)^{-\frac{p_{n}-2}{2}}.$
Then, there exists $k\in \{1,2\}$, and sequences of points $\{O_{n}^{i}\}$ for $i\in\{1,k\}$, such that there exists some $C>0$, independent of $n$, satisfying the following condition:
\begin{equation}\label{D22}
  \Theta_n^{k}(x):=\big(d_n^k(x)\big)^{\frac{2}{p_n-2}(2-\frac{1}{N})}(\hat{\epsilon}_n)^{-\frac{2}{p_n-2}(1-\frac{1}{N})}u_{p_n}(x)\leq C,\quad \forall x\in \Omega,~~ \forall n\in \mathbb{N},
\end{equation}
where $d_{n}^{k}(x):=\min\{|x-O_{n}^{i}|:i=1,\cdots,k\}$.
\end{theorem}
To prove Theorem \ref{L-criti-blow1}, we first analyze the local behavior of the sequence $\{u_{p_{n}}\}$ near its global maximum point.
\begin{lemma}\label{L-criti-blow2} Let $\{O_{n}\}\subset\Omega$ be such that $u_{p_n}(O_{n})=\|u_{p_{n}}\|_{\infty}$.
Then, up to a subsequence, the following properties hold:
\begin{description}
  \item[(i)] $\lambda_{n}(\hat{\epsilon}_{n})^{2}\to 0$ as $n\to+\infty.$
  \item[(ii)] We define
\begin{equation*}
  \hat{U}_{p_n,1}(y):=(\hat{\epsilon}_{n})^{\frac{2}{p_{n}-2}}u_{p_{n}}(O_{n}+\hat{\epsilon}_{n}y), \quad \forall y\in\hat{\Omega}_{p_{n}}:=\frac{\Omega-O_{n}}{\hat{\epsilon}_{n}}.
\end{equation*}
Then
$\hat{U}_{p_n,1}\rightarrow \hat{U}_{1}$ in $C^{1}$ uniformly on compact subsets of ~$\mathbb{R}^{N}$, where $\hat{U}_{1}$ satisfies $m(\hat{U}_{1})\leq 2$ and solves
\begin{equation}\label{10-3-1}
\left\{
\begin{aligned}
 & -\Delta \hat{U}_{1}=\hat{U}_{1}^{2^{*}-1} \quad \mbox{in} \ \mathbb{R}^{N},\\
 & \hat{U}_{1}(0)=\max\limits_{y\in\mathbb{R}^{N}}\hat{U}_{1}(y).
 \end{aligned}
 \right.
\end{equation}
  \item[(iii)]  There exists $\hat{\phi}_{n,1}\in C^{\infty}_{0}(\Omega)$ such that for sufficiently large $n$, $supp\hat{\phi}_{n,1}\subset B_{\hat{R}_1\hat{\epsilon}_{n}}(O_{n})$ for some $\hat{R}_1>0$ and
\begin{equation*}
Q_{\lambda_{n},u_{p_{n}}}(\hat{\phi}_{n,1})<0.
\end{equation*}
\end{description}
\end{lemma}

\begin{proof}
(i)~ From (\ref{C1-Lapla}) and Lemma \ref{lower bound}, it follows that there exists $\bar{\lambda}\in[0,1]$ such that
\begin{eqnarray*}
\lambda_{n}(\hat{\epsilon}_{n})^{2}\to \bar{\lambda} ~~\mbox{as}~~ n\to+\infty.
\end{eqnarray*}

Now, we demonstrate that $\bar \lambda=0.$ Suppose, for contradiction, that $\bar \lambda \in (0,1].$ Define
\begin{equation*}
v_{p_{n}}(y):=\lambda_{n}^{-\frac{1}{p_{n}-2}}u_{p_{n}}(O_{n}+\lambda_{n}^{-\frac{1}{2}}y),\quad \forall y\in\check{\Omega}_{p_{n}}:=\lambda_{n}^{\frac{1}{2}}(\Omega-O_{n}).
\end{equation*}
Then $v_{p_{n}}$ satisfies
\begin{displaymath}
\left\{
\begin{aligned}
&-\Delta v_{p_{n}} + v_{p_{n}} = v_{p_{n}}^{p_{n}-1} \quad  &\mbox{in}& \  \check{\Omega}_{p_{n}},\\
& v_{p_{n}}(y) \leq v_{p_{n}}(0)=(\lambda_{n}(\hat{\epsilon}_{n})^{2})^{-\frac{1}{p_{n}-2}} \quad &\mbox{in}& \ \check{\Omega}_{p_{n}},\\
&v_{p_{n}}=0 \quad &\mbox{on}& \ \partial\check{\Omega}_{p_{n}}.
\end{aligned}
\right.
\end{displaymath}
Furthermore, $v_{p_{n}}\rightarrow v_{*}$ in $C^{1}$ uniformly on compact subsets of $\check{H}_{*}$, where $v_{*}$ solves
\begin{eqnarray}\label{10-3-2}
\left\{
\begin{aligned}
&-\Delta v_{*} + v_{*} =v_{*}^{2^{*}-1} \quad  &\mbox{in}& \ \check{H}_{*},\\
& v_{*}(y) \leq v_{*}(0)=\bar{\lambda}^{-\frac{1}{2^{*}-2}} \quad &\mbox{in}& \ \check{H}_{*},\\
&v_{*}=0 \quad &\mbox{on}& \ \partial\check{H}_{*}
\end{aligned}
\right.
\end{eqnarray}
with
\begin{equation*}
\check{H}_{*}=\left\{
\begin{aligned}
&\mathbb{R}^{N} \quad &~&\mbox{if} \ \lambda_{n}^{\frac{1}{2}}dist(O_{n},\partial\Omega)\rightarrow +\infty,\\
&\mathbb{R}^{N}_{L_1} \quad &~& \mbox{if} \ \lambda_{n}^{\frac{1}{2}}dist(O_{n},\partial\Omega)\rightarrow L_{1},
\end{aligned}
\right.
\end{equation*}
where $\mathbb{R}^{N}_{L_1}:=\{x=(x_{1}, \cdots, x_{N})\in \mathbb{R}^N: x_{N}>-L_{1}\}$.
By elliptic regularity up to the boundary, we obtain that $\{|\nabla v_{p_{n}}|\}$ is uniformly bounded. By standard arguments as in \cite{GS}, we deduce $L_{1}>0.$
Furthermore, we can verify that $m(v_{*})\leq 2$. Indeed, if $\zeta_1, \cdots, \zeta_k$ are orthogonal in $L^{2}(\check{H}_{*})$ and satisfy
\begin{equation*}
  \int_{\check{H}_{*}}|\nabla\zeta_i|^2 dy+\int_{\check{H}_{*}}\left(1-(2^{*}-1) v_{*}^{2^*-2}\right) \zeta_i^2 dy< 0, \forall i=1,\cdots,k,
\end{equation*}
then, by setting $\zeta_{i,n}(x)=\lambda_n^{\frac{N-2}{4}}\zeta_{i}(\lambda_n^{\frac{1}{2}}(x-O_n))$, we observe that
$\zeta_{i,n}$ are orthogonal in $L^{2}(\Omega)$ and
\begin{align*}
  &\int_{\Omega}|\nabla\zeta_{i,n}|^2dx+\int_{\Omega}\left(\lambda_n-(p_n-1)u_{p_n}^{p_n-2}\right)\zeta_{i,n}^2dx\nonumber\\
=&\int_{ \check{\Omega}_{p_{n}}}|\nabla\zeta_{i}|^2dy+\int_{ \check{\Omega}_{p_{n}}}\left(1-(p_n-1)\lambda_n^{-1}u_{p_n}^{p_n-2}(O_n+\lambda_n^{-\frac{1}{2}}y)\right)\zeta_{i}^2dy\nonumber\\
=&\int_{ \check{\Omega}_{p_{n}}}|\nabla\zeta_{i}|^2dy+\int_{ \check{\Omega}_{p_{n}}}\left(1-(p_n-1)v_{p_n}^{p_n-2}\right)\zeta_{i}^2dy\nonumber\\
\rightarrow&\int_{\check{H}_{*}}|\nabla\zeta_i|^2 dy+\int_{\check{H}_{*}}\left(1-(2^{*}-1) v_{*}^{2^*-2}\right) \zeta_i^2 dy< 0\label{10-3-3}
\end{align*}
as $n\rightarrow+\infty$ for all $i=1,\cdots,k$. Hence, $m(v_*)\leq m(u_{p_n})\leq 2$. 
This implies that $\nabla v_* \in L^2(\check{H}_{*})$ and $\frac{1}{2^*}|v_*|^{2^*}-\frac{1}{2}v_*^2\in L^1(\check{H}_{*}).$ If $\check{H}_{*}=\mathbb{R}^N$, then 
by the Pohozaev identity, problem (\ref{10-3-2}) admits only the trivial solution. If $\check{H}_{*}$ is a half-space, by Theorem I.1 in \cite{EL1982} (see also Theorem 1.25 in \cite{P2010}), problem (\ref{10-3-2}) again admits only the trivial solution.
In either case, a contradiction arises. Therefore, we conclude that $\bar\lambda=0.$

(ii)~ Clearly, $\hat{U}_{p_n,1}$ satisfies
\begin{displaymath}
\left\{
\begin{aligned}
&-\Delta \hat{U}_{p_n,1}+\lambda_{n}(\hat{\epsilon}_{n})^{2}\hat{U}_{p_n,1}=\hat{U}_{p_n,1}^{p_{n}-1}\quad  &\mbox{in}& \ \hat{\Omega}_{p_{n}},\\
& \hat{U}_{p_n,1}(y)\leq \hat{U}_{p_n,1}(0)=1 \quad &\mbox{in}& \ \hat{\Omega}_{p_{n}},\\
&\hat{U}_{p_n,1}=0\quad &\mbox{on}& \   \partial\hat{\Omega}_{p_{n}},
\end{aligned}
\right.
\end{displaymath}
and $\hat{U}_{p_n,1}\rightarrow \hat{U}_1$ in $C^{1}$ uniformly on compact subsets of $\hat{H}_{*}$, where $\hat{U}_{1}$ satisfies $m(\hat{U}_{1})\leq 2$ and solves
\begin{equation}\label{10-3-4}
\left\{
\begin{aligned}
&-\Delta \hat{U}_{1}=\hat{U}_{1}^{2^{*}-1}  \quad  &\mbox{in}& \ \hat{H}_{*},\\
& \hat{U}_{1}(y)\leq \hat{U}_{1}(0)=1 \quad &\mbox{in}&  \ \hat{H}_{*},\\
& \hat{U}_{1}=0\quad &\mbox{on}& \ \partial\hat{H}_{*}.
\end{aligned}
\right.
\end{equation}
Here $\hat{H}_{*}$ satisfies
\begin{equation*}
\hat{\Omega}_{p_{n}}\to \hat{H}_{*}=\left\{
  \begin{aligned}
&\mathbb{R}^{N}  \quad &\mbox{if}& \ \frac{dist(O_{n},\partial\Omega)}{\hat{\epsilon}_{n}}\rightarrow +\infty,\\
&\mathbb{R}^{N}_{L_2} \quad &\mbox{if}& \ \frac{dist(O_{n},\partial\Omega)}{\hat{\epsilon}_{n}}\rightarrow L_2
\end{aligned}
\right.
\end{equation*}
with $\mathbb{R}^{N}_{L_2}:=\{x=(x_{1}, \cdots, x_{N})\in \mathbb{R}^N: x_{N}>-L_{2}\}$ for some $L_2> 0.$ By Theorem 12 in \cite{F2007}, problem (\ref{10-3-4}) has only the trivial solution when $\hat{H}_{*}=\mathbb{R}^{N}_{L_2}.$ Therefore, we deduce $\hat{H}_{*}=\mathbb{R}^{N}$ and then (\ref{10-3-1}) holds.

(iii)~ By (\ref{10-3-4}), it is easily seen that $ Q_{0,\hat{U}_{1}}(\hat{U}_{1})<0$,
which implies that $m(\hat{U}_{1})\geq 1$.  Consequently, there exists a function $\hat{\phi}_{1}\in C^{\infty}_{0}(\mathbb{R}^{N})$ with $supp\hat{\phi}_{1}\subset B_{\hat{R}_1}$ for some $\hat{R}_1>0$ such that $Q_{0,\hat{U}_{1}}(\hat{\phi}_1)<0.$
We define the scaled function
\begin{eqnarray*}
\hat{\phi}_{n,1}(x):=(\hat{\epsilon}_{n})^{-\frac{N-2}{2}}\hat{\phi}_{1}\left(\frac{x-O_{n}}{\hat{\epsilon}_{n}}\right),~~\forall x\in \mathbb{R}^N.
\end{eqnarray*}
Clearly, $supp\hat{\phi}_{n,1}\subset B_{\hat{R}_1\hat{\epsilon}_{n}}(O_{n})$. 
Note that $B_{\hat{R}_1\hat{\epsilon}_{n}}(O_{n})\subset\Omega$ for sufficiently large $n$,
we have $\hat{\phi}_{n,1}\in C^{\infty}_{0}(\Omega)$. Then
\begin{align*}
  Q_{\lambda_{n},u_{p_{n}}}(\hat{\phi}_{n,1})
=&\int_{\Omega}|\nabla\hat{\phi}_{n,1}|^2dx+\int_{\Omega}\left(\lambda_n-(p_n-1)u_{p_n}^{p_n-2}\right)\hat{\phi}_{n,1}^2dx\nonumber\\
=&\int_{\hat{\Omega}_{p_{n}}}|\nabla\hat{\phi}_{1}|^2dy+\int_{ \hat{\Omega}_{p_{n}}}\left(\lambda_n(\hat{\epsilon}_{n})^2-(p_n-1)\hat{U}_{p_n,1}^{p_n-2}\right)\hat{\phi}_{1}^2dy\nonumber\\
\rightarrow&\int_{\mathbb{R}^N}|\nabla\hat{\phi}_1|^2 dy-(2^{*}-1)\int_{\mathbb{R}^N}\hat{U}_{1}^{2^*-2} \hat{\phi}_1^2 dy=Q_{0,\hat{U}_{1}}(\hat{\phi}_1)<0.
\end{align*}
\end{proof}

\begin{proof}[Proof of Theorem \ref{L-criti-blow1}]

The proof proceeds in two steps.

{\it Step 1. (\ref{D22}) holds for $k=1$ or there exists $\{\hat{\phi}_{n,2}\}\subset C^{\infty}_{0}(\Omega)$ with $supp\hat{\phi}_{n,2}\subset B_{\hat{R}_2\hat{\epsilon}_{n}^{(2)}}(O_{n}^{2})$ for some $\hat{R}_2>0$ such that for sufficiently large $n$,
\begin{equation}\label{9-30-5}
Q_{\lambda_{n}, u_{p_{n}}}(\hat{\phi}_{n,2})<0.
\end{equation}
}

Let $O_n^1=O_n$ and $\hat{\epsilon}_{n}^{(1)}=\hat{\epsilon}_n$.
If (\ref{D22}) does not hold for $k=1$, then there exists $\{O_n^2\}\subset\Omega$ such that
\begin{equation}\label{Y-J-0}
\max\limits_{x\in\Omega}\Theta_n^1(x)=\Theta_n^1(O_n^2)
=\big(d_n^1(O_n^2)\big)^{\frac{2}{p_n-2}(2-\frac{1}{N})}(\hat{\epsilon}_n^{(1)})^{-\frac{2}{p_n-2}(1-\frac{1}{N})}u_{p_n}(O_n^2)\to +\infty.
\end{equation}
Clearly, $O_n^2\neq O_n^1$. Define 
$$
A_n(x):=(\hat{\epsilon}_n^{(1)})^{-\frac{2}{p_n-2}(1-\frac{1}{N})}u_{p_n}(x),~ \forall x\in\Omega.
$$
 By (\ref{Y-J-0}), it follows that $A_n(O_n^2)\to+\infty$. Set
\begin{equation*}
 \hat{\epsilon}_n^{(2)}:=A_n(O_n^2)^{-\frac{p_n-2}{2(2-\frac{1}{N})}}.
\end{equation*}
It is easy to observe that
\begin{equation}\label{D11}
\hat{\epsilon}_n^{(2)}\geq \hat{\epsilon}_n^{(1)}, \quad \hat{\epsilon}_n^{(2)}\to 0\quad \mbox{and}\quad \frac{|O_n^2-O_n^1|}{\hat{\epsilon}_n^{(2)}}\to+\infty.
\end{equation}
Define
\begin{equation}\label{D13}
  \hat{U}_{p_n,2}(y):=(\hat{\epsilon}_{n}^{(2)})^{\frac{4}{p_{n}-2}(1-\frac{1}{2N})}A_{n}(O_{n}^2+\hat{\epsilon}_{n}^{(2)} y), \quad \forall y\in\hat{\Omega}_{p_{n}}^{2}:=\frac{\Omega-O_{n}^2}{\hat{\epsilon}_{n}^{(2)}}.
\end{equation}
By direct computations, we have
\begin{equation*}\label{D26-Lapla}
  \left\{
  \begin{aligned}
  & -\Delta\hat{U}_{p_n,2}+\lambda_n(\hat{\epsilon}_n^{(2)})^{2}\hat{U}_{p_n,2}=\left(\frac{\hat{\epsilon}_n^{(1)}}{\hat{\epsilon}_n^{(2)}}\right)^{2(1-\frac{1}{N})}\hat{U}_{p_n,2}^{p_n-1} \quad &\mbox{in}& \  \hat{\Omega}_{p_{n}}^{2},\\
  &\hat{U}_{p_n,2}(0)=1, \quad\hat{U}_{p_n,2}=0 \quad &\mbox{on}& \  \partial\hat{\Omega}_{p_{n}}^{2}.
  \end{aligned}
  \right.
\end{equation*}
To derive (\ref{9-30-5}), we seek to show that $\hat{U}_{2}$ satisfies
\begin{equation}\label{Y-J-5}
\left\{
\begin{aligned}
 &-\Delta \hat{U}_2=\hat{\delta} \hat{U}_2^{2^*-1} \quad  \mbox{in} \ \mathbb{R}^{N},\\
&\hat{U}_2(0)=1
\end{aligned}
\right.
\end{equation}
for some $\hat{\delta}\in(0,1].$
Indeed, if (\ref{Y-J-5}) holds, we can verify that $Q_{0,\hat{U}_{2}}(\hat{U}_{2})<0$,
which implies that $m(\hat{U}_{2})\ge 1$. Consequently, there exists $\hat{\phi}_2\in C^{\infty}_{0}(\mathbb{R}^{N})$ such that $supp\hat{\phi}_2\subset B_{\hat{R}_2}$ for some $\hat{R}_2>0$ and $ Q_{0,\hat{U}_2}(\hat{\phi}_2)<0$.
Set 
$$
\hat{\phi}_{n,2}(x):=(\hat{\epsilon}_{n}^{(2)})^{-\frac{N-2}{2}}\hat{\phi}_{2}\left(\frac{x-O_{n}^2}{\hat{\epsilon}_{n}^{(2)}}\right),~~\forall x\in \mathbb{R}^N.
$$
Clearly, $supp\hat{\phi}_{n,2}\subset B_{\hat{R}_2\hat{\epsilon}_{n}^{(2)}}(O_{n}^{2})$.
  Using similar arguments as in Lemma \ref{L-criti-blow2} (iii), we deduce that for sufficiently large $n$, $\hat{\phi}_{n,2}\in C^{\infty}_{0}(\Omega)$, and then (\ref{9-30-5}) holds.

To prove (\ref{Y-J-5}), we first establish the following result:
\begin{equation}\label{Y-J-3}
 \lambda_n(\hat{\epsilon}_n^{(2)})^2\to 0.
\end{equation}
To derive (\ref{Y-J-3}), we need to rule out the following two cases.

Case (i). $\lambda_n(\hat{\epsilon}_n^{(2)})^2\to\tilde{\lambda}\in(0,+\infty)$. We divide it into the following two subcases.

(a). $O_n^2\to O^2\in\Omega$. 
In this case, we have
$$
\frac{dist(O_n^{2},\partial\Omega)}{\hat{\epsilon}_n^{(2)}}\to +\infty,
$$ 
which implies that
$\hat{\Omega}_{p_{n}}^{2}\to\mathbb{R}^N$.
By (\ref{Y-J-0}) and (\ref{D13}), we deduce
\begin{align}
\hat{U}_{p_n,2}(y)=\frac{u_{p_n}(O_{n}^2+\hat{\epsilon}_{n}^{(2)}y)}{u_{p_n}(O_{n}^{2})}
\leq \left(\frac{d_n^1(O_n^2)}{d_n^1(O_{n}^2+\hat{\epsilon}_{n}^{(2)}y)}\right)^{\frac{4}{p_n-2}(1-\frac{1}{2N})}
=\left(\frac{|O_n^2-O_n^1|}{|O_{n}^2+\hat{\epsilon}_{n}^{(2)}y-O_n^1|}\right)^{\frac{4}{p_n-2}(1-\frac{1}{2N})}.\label{D5}
\end{align}
Observe that for any  $y\in \hat{\Omega}_{p_{n}}^{2},$ 
\begin{equation*}
|O_{n}^2-O_n^1|-\hat{\epsilon}_{n}^{(2)}|y|\leq |O_{n}^2+\hat{\epsilon}_{n}^{(2)}y-O_n^1|\leq |O_{n}^2-O_n^1|+\hat{\epsilon}_{n}^{(2)}|y|.
\end{equation*}
Combining this with (\ref{D11}), we conclude that for any bounded subset $D\subset\mathbb{R}^{N},$
\begin{equation}\label{c}
\frac{|O_{n}^2+\hat{\epsilon}_{n}^{(2)}y-O_n^1|}{|O_n^2-O_n^1|}\to 1, \ \forall y\in D.
\end{equation}
Hence, $\hat{U}_{p_n,2}\to \hat{U}_2$ uniformly in $C^{1}$ on compact subsets of $\mathbb{R}^{N}$. By (\ref{D11}), we have $\frac{\hat{\epsilon}_n^{(1)}}{\hat{\epsilon}_n^{(2)}}\in(0,1]$, and thus
$\hat{U}_2$ satisfies
\begin{equation}\label{D6}
  \left\{
  \begin{aligned}
  & -\Delta \hat{U}_{2}+\tilde{\lambda} \hat{U}_{2}=\hat{\delta} \hat{U}_{2}^{\frac{N+2}{N-2}} \quad &\mbox{in}& \ \mathbb{R}^{N},\\
  & \hat{U}_{2}(0)=1,  
  \end{aligned}
  \right.
\end{equation}
where 
\begin{equation}\label{D18}
\hat{\delta}:=\lim\limits_{n\to+\infty}\left(\frac{\hat{\epsilon}_n^{(1)}}{\hat{\epsilon}_n^{(2)}}\right)^{2(1-\frac{1}{N})}\in[0,1].
\end{equation}
We claim that $\hat{\delta}>0$. In fact, if not, we would assume $\hat{\delta}=0$. From (\ref{D5})-(\ref{D6}), we have $\hat{U}_{2}(0)=\max\limits_{x\in \mathbb{R}^N}\hat{U}_{2}(x)$, which implies $-\Delta \hat{U}_{2}(0)\ge0$. Given the fact that $\tilde \lambda>0$, this leads to a contradiction.

Now we show that $m(\hat{U}_{2})\leq 2$. It suffices to prove that if $m(\hat{U}_2)\geq 3$, then $m(u_{p_n})\geq 3$.  We consider functions $\xi_1,\xi_2,\xi_3\in C^{\infty}_{0}(\mathbb{R}^{N})$ such that $supp\xi_i\subset B_{R_i}(Q_{i})\subset\Omega$ for some $R_i>0$, and $B_{R_i}(Q_i)\cap B_{R_j}(Q_{j})=\emptyset$ for $i\neq j$.
Assume that the following hold
\begin{equation}\label{Y-J-7}
  \int_{\mathbb{R}^{N}}|\nabla\xi_i|^2 dy + \int_{\mathbb{R}^{N}}\left(\tilde{\lambda} -(2^*-1)\hat{\delta}\hat{U}_{2}^{\frac{4}{N-2}} \right)\xi_i^2dy<0, \ \forall i=1,2,3.
\end{equation}
For each $i$, we define
$$
\xi_{i,n}(x):=\left(\hat{\epsilon}_n^{(2)}\right)^{-\frac{N-2}{2}}\xi_{i}\left(\frac{x-O_n^2}{\hat{\epsilon}_n^{(2)}}\right),~~\forall x\in \mathbb{R}^N.
$$
Clearly, $supp\xi_{i,n}\subset I_{i,n}:=B_{R_i\hat{\epsilon}_n^{(2)}}(O_n^2+\hat{\epsilon}_n^{(2)}Q_{i})$.
Since $\hat{\Omega}_{p_{n}}^{2}\to\mathbb{R}^{N}$, 
we deduce that $I_{i,n}\subset\Omega$ for sufficiently large $n$.
Furthermore, 
$I_{i,j,n}:=I_{i,n}\cap I_{j,n}=\emptyset.$ Then we have
\begin{equation*}
\begin{aligned}
\int_{\Omega}\xi_{i,n}(x)\xi_{j,n}(x)dx
=&\int_{I_{i,j,n}}\xi_{i,n}(x)\xi_{j,n}(x)dx=0, \quad i\neq j.
\end{aligned}
\end{equation*}
It follows that $\xi_{1,n},\xi_{2,n},\xi_{3,n}$ are orthogonal in $L^{2}(\Omega)$.
Hence, by (\ref{Y-J-7}), we obtain
\begin{align*}
  &\int_{\Omega}|\nabla\xi_{i,n}|^2dx+\int_{\Omega}\left(\lambda_n-(p_n-1)u_{p_n}^{p_n-2}\right)\xi_{i,n}^2dx\nonumber\\
=&\int_{B_{R_i}(Q_{i})}|\nabla\xi_{i}|^2dy+\int_{B_{R_i}(Q_{i})}\left(\lambda_n(\hat{\epsilon}_n^{(2)})^{2}-(p_n-1)(\hat{\epsilon}_n^{(2)})^{2}u_{p_n}^{p_n-2}(O_n^2+\hat{\epsilon}_{n}^{(2)}y)\right)\xi_{i}^2dy\nonumber\\
=&\int_{\mathbb{R}^{N}}|\nabla\xi_{i}|^2dy+\int_{ \mathbb{R}^{N}}\left(\lambda_n(\hat{\epsilon}_n^{(2)})^{2}-(p_n-1)\left(\frac{\hat{\epsilon}_n^{(1)}}{\hat{\epsilon}_n^{(2)}}\right)^{2(1-\frac{1}{N})}\hat{U}_{p_n,2}^{p_n-2}\right)\xi_{i}^2dy\nonumber\\
\rightarrow&\int_{\mathbb{R}^{N}}|\nabla\xi_i|^2 dy  + \int_{\mathbb{R}^{N}}\left(\tilde{\lambda} -(2^*-1)\hat{\delta}\hat{U}_{2}^{\frac{4}{N-2}} \right)\xi_i^2dy<0, \ \forall i=1,2,3, \label{10-3-3}
\end{align*}
which implies that $m(u_{p_n})\geq 3$.  Consequently, we arrive at a contradiction. Therefore, we conclude that
$m(\hat{U}_2)\leq 2$.

 Setting $\hat{V}_2:=\hat{\delta}^{\frac{N-2}{4}}\hat{U}_2,$ we obtain
\begin{equation*}\label{D7}
  \left\{
  \begin{aligned}
  & -\Delta \hat{V}_{2}+\tilde{\lambda} \hat{V}_{2}=\hat{V}_{2}^{\frac{N+2}{N-2}} \quad &\mbox{in}& \ \mathbb{R}^{N},\\
  & \hat{V}_{2}(0)=\hat{\delta}^{\frac{N-2}{4}}.  
  \end{aligned}
  \right.
\end{equation*}
Since $m(\hat{V}_2)=m(\hat{U}_2)\leq 2$, by employing arguments similar as in Lemma \ref{L-criti-blow2}, we derive that the above equation admits only a trivial solution in $\mathbb{R}^{N}$, thereby yielding a contradiction.

(b). $O_n^2\to O^2\in\partial\Omega$. If $\frac{dist(O_n^{2},\partial\Omega)}{\hat{\epsilon}_n^{(2)}}\to +\infty$, then we can derive a contradiction as in case (a). Therefore, we now assume that
\begin{equation}\label{D14}
\frac{dist(O_n^{2},\partial\Omega)}{\hat{\epsilon}_n^{(2)}}\to L_3
\end{equation}
for some $L_3\geq 0$. By applying a space rotation, we can assume that the tangent space to $\partial\Omega$ at $O^2$ is parallel to the plane $\{x_N=0\}$.
For simplicity, we assume that $O^2$ is the origin and $(0,0,\cdots.-1)$ is the outer normal to $\partial\Omega$ at $O^2$. Since $\partial\Omega$ is smooth, there exists $r_0>0$ and a smooth function $\hat\gamma:\mathbb{R}^{N-1}\to \mathbb{R}$ such that
\begin{itemize}
  \item [(i)] $\hat\gamma(0)=0, \frac{\partial \hat\gamma}{\partial x_i}(0)=0$ with $i=1,\cdots,N-1,$ 
  \item [(ii)] $\Omega\cap B_{r_0}=\{(x',x_N)\in \mathbb{R}^N: x_N>\hat\gamma(x')\},\  \partial\Omega\cap B_{r_0}=\{(x',x_N)\in \mathbb{R}^N: x_N=\hat\gamma(x')\},$
\end{itemize}
where $x':=(x_1,\cdots,x_{N-1}).$
For $x=(x_1,\cdots,x_N)\in \Omega\cap B_{r_0},$ we define a map $y=\Phi(x)=(\Phi_1(x),\cdots,\Phi_N(x))$ by
\begin{equation*}
\left\{
  \begin{aligned}
  &\Phi_i(x)=x_i, \quad i=1,2,\cdots,N-1,\\
  &\Phi_N(x)=x_N-\hat\gamma(x').
  \end{aligned}
  \right.
\end{equation*}
From (ii), we deduce that $\Phi(\Omega\cap B_{r_0})\subset\mathbb{R}^{N}_{+}.$
Furthermore, by (i), we know that $\Phi$ has the inverse mapping $\Psi=\Phi^{-1}:B_{2\delta}^{+}\to \Omega\cap B_{\sigma}$ for some $\delta>0$ and $\sigma\in(0,r_0).$ Therefore, for any $y\in B_{2\delta}^{+}$, there exists $x\in \Omega\cap B_{\sigma}$ such that $x=\Psi(y),$ where $\Psi(y)=(\Psi_1(y),\Psi_2(y),\cdots,\Psi_N(y)).$

Set $w_{p_n}(y):=u_{p_n}(x)=u_{p_n}(\Psi(y)), y\in B_{2\hat{\delta}}^{+}$. By direct computation, $w_{p_n}$ satisfies
\begin{equation}\label{D12}
  \left\{
  \begin{aligned}
  &-\left(\sum_{i,j}a_{i,j}(y)(w_{p_n})_{y_{i}y_{j}}+\sum_{j}b_{j}(y)(w_{p_n})_{y_j}\right)+\lambda_n w_{p_n}=w_{p_n}^{p_n-1} \quad \mbox{in} \ B_{2\delta}^{+},\\
  & w_{p_n}=0 \quad \mbox{on} \  B_{2\delta}\cap \{y=(y',y_N):y_N=0\},
  \end{aligned}
  \right.
\end{equation}
and $m(w_{p_n})\leq 2$, where $a_{i,j}(y)=\sum_{k}(\Phi_{i})_{x_k}(\Psi(y))(\Phi_{j})_{x_k}(\Psi(y))$ and $b_{j}=(\Delta_{x}\Phi)(\Psi(y))$.

 We denote $O_n^2=((O_n^2)', O_{n,N}^2)$, where $(O_n^2)'\in\mathbb{R}^{N-1}$ and $O_{n,N}^2\in \mathbb{R}$.
By 
the definition of $\hat\gamma$, we have $\Phi(O^2)=\Phi(0)=0$,
and $\Phi(O_n^2)=((O_n^2)', l_n)$, where $l_n:=O_{n,N}^2-\hat\gamma((O_n^2)')$. 
Clearly, $|(O_n^2)'|\leq \delta$ and $0\leq l_n\leq \delta.$ 
Furthermore, we claim that
\begin{equation}\label{D16}
  \frac{l_n}{\hat{\epsilon}_n^{(2)}}\to l_0 \ \mbox{for} \ \mbox{some} \ l_0\geq 0.
\end{equation}
To establish this, we first take $M_n=(M_n',\hat\gamma(M_n'))\in\partial\Omega$ such that $dist(O_n^{2},\partial\Omega)=|O_n^{2}-M_n|$. Then, by (\ref{D14}), (i), the Mean Value theorem and the elementary inequality $(a-b)^2\ge \frac{a^2}{2}-b^2$, we obtain
\begin{equation}\label{D15}
2 L_3 \geq \frac{|O_n^{2}-M_n|}{\hat{\epsilon}_n^{(2)}}
=\frac{(|(O_n^2)'-M_n'|^2+|l_n+\hat\gamma((O_n^2)')-\hat\gamma(M_n') |^2)^{\frac{1}{2}}}{\hat{\epsilon}_n^{(2)}}\geq \frac{l_n}{2\hat{\epsilon}_n^{(2)}}
\end{equation}
for sufficiently large $n$. Hence, in view of
 $l_n,\hat{\epsilon}_n^{(2)}\geq 0$, we deduce that (\ref{D16}) holds.

We define
\begin{equation*}
  W_{p_n}(z):=(\hat{\epsilon}_n^{(1)})^{-\frac{2}{p_n-2}(1-\frac{1}{N})}(\hat{\epsilon}_{n}^{(2)})^{\frac{4}{p_{n}-2}(1-\frac{1}{2N})}w_{p_{n}}((O_n^2)'+\hat{\epsilon}_{n}^{(2)}z', \hat{\epsilon}_{n}^{(2)}z_N),~~ \forall z\in B_{\frac{\delta}{\hat{\epsilon}_n^{(2)}}}^{+}.
\end{equation*}
Clearly, $W_{p_n}$ satisfies
\begin{equation}\label{D10}
  \left\{
  \begin{aligned}
  &-\left(\sum_{i,j}\hat a_{n,i,j}(z)(W_{p_n})_{z_{i}z_{j}}+\hat{\epsilon}_n^{(2)}\sum_{j}\hat b_{n,j}(z)(W_{p_n})_{z_j}\right)+\lambda_n(\hat{\epsilon}_n^{(2)})^2 W_{p_n}=\left(\frac{\hat{\epsilon}_n^{(1)}}{\hat{\epsilon}_n^{(2)}}\right)^{2(1-\frac{1}{N})}W_{p_n}^{p_n-1} \quad \mbox{in} \ B_{\frac{\delta}{\hat{\epsilon}_n^{(2)}}}^{+},\\
  & W_{p_n}=0 \quad \mbox{on} \ B_{\frac{\delta}{\hat{\epsilon}_n^{(2)}}}\cap \{z=(z',z_N):z_N=0\},
  \end{aligned}
  \right.
\end{equation}
where $\hat a_{n, i,j}(z)=a_{i,j}((O_n^2)'+\hat{\epsilon}_n^{(2)}z', \hat{\epsilon}_n^{(2)}z_N)$
and $\hat b_{n,j}(z)=b_j((O_n^2)'+\hat{\epsilon}_n^{(2)}z',\hat{\epsilon}_n^{(2)}z_N)$.
By the definition of $w_{p_n}$ and (\ref{Y-J-0}), we deduce that
\begin{equation*}
\begin{aligned}
  W_{p_n}(z)&=(\hat{\epsilon}_n^{(1)})^{-\frac{2}{p_n-2}(1-\frac{1}{N})}(\hat{\epsilon}_{n}^{(2)})^{\frac{4}{p_{n}-2}(1-\frac{1}{2N})}u_{p_{n}}\left((O_n^2)'+\hat{\epsilon}_{n}^{(2)}z', \hat{\epsilon}_{n}^{(2)}z_N+\hat\gamma\big((O_n^2)'+\hat{\epsilon}_{n}^{(2)}z'\big)\right)\\
  &=\frac{u_{p_{n}}\left((O_n^2)'+\hat{\epsilon}_{n}^{(2)}z', \hat{\epsilon}_{n}^{(2)}z_N+\hat\gamma\big((O_n^2)'+\hat{\epsilon}_{n}^{(2)}z'\big)\right)}{u_{p_n}(O_n^2)}\\
  &\leq \left(\frac{d_n^1(O_n^2)}{d_n^1\left((O_n^2)'+\hat{\epsilon}_{n}^{(2)}z', \hat{\epsilon}_{n}^{(2)}z_N+\hat\gamma\big((O_n^2)'+\hat{\epsilon}_{n}^{(2)}z'\big)\right)}\right)^{\frac{4}{p_n-2}(1-\frac{1}{2N})}\\
  &=\left(\frac{\left|O_n^2-O_{n}^{1}\right|^2}{\left|(O_n^2)'+\hat{\epsilon}_{n}^{(2)}z'-(O_{n}^{1})'\right|^2+\left| \hat{\epsilon}_{n}^{(2)}z_N+\hat\gamma\big((O_n^2)'+\hat{\epsilon}_{n}^{(2)}z'\big)-O_{n,N}^{1}\right|^2}\right)^{\frac{2}{p_n-2}(1-\frac{1}{2N})}, ~\quad \forall z\in B_{\frac{\delta}{\hat{\epsilon}_n^{(2)}}}^{+},
  \end{aligned}
\end{equation*}
where $O_n^1=((O_{n}^{1})',O_{n,N}^{1})$ with $(O_{n}^{1})'\in\mathbb{R}^{N-1}.$ By (\ref{D15}), $\frac{\partial\hat{\gamma}}{\partial x_i}(0)=0$ and the Taylor expansion, we have
\begin{equation*}
\begin{aligned}
&\left|(O_n^2)'+\hat{\epsilon}_{n}^{(2)}z'-(O_{n}^{1})'\right|^2+\left| \hat{\epsilon}_{n}^{(2)}z_N+\hat\gamma\big((O_n^2)'+\hat{\epsilon}_{n}^{(2)}z'\big)-O_{n,N}^{1}\right|^2\\
\geq &\frac{1}{2}\left|(O_n^2)'-(O_{n}^{1})'\right|^2-(\hat{\epsilon}_{n}^{(2)})^2|z'|^2+\frac{1}{2}\left|O_{n,N}^2-O_{n,N}^{1}+ \hat{\epsilon}_{n}^{(2)}z_N+\hat\gamma'\big((O_n^2)'\big)\hat{\epsilon}_{n}^{(2)}z'\right|^2-l_n^2+o_n(1)\\
\geq &\frac{1}{2}\left|(O_n^2)'-(O_{n}^{1})'\right|^2-(\hat{\epsilon}_{n}^{(2)})^2|z'|^2+\frac{1}{4}\left|O_{n,N}^2-O_{n,N}^{1}\right|^2
-\frac{1}{2}\left|\hat{\epsilon}_{n}^{(2)}z_N+\hat\gamma'\big((O_n^2)'\big)\hat{\epsilon}_{n}^{(2)}z'\right|^2-16L_3^2(\hat{\epsilon}_{n}^{(2)})^2+o_n(1)\\
\geq &\frac{1}{4}\left|O_n^2-O_n^1\right|^2-\left(1+\left|\hat\gamma'\big((O_n^2)'\big)\right|^2\right)(\hat{\epsilon}_{n}^{(2)})^2|z|^2-16L_3^2(\hat{\epsilon}_{n}^{(2)})^2+o_n(1).
\end{aligned}
\end{equation*}
Here we have used the fact that $l_n=O_{n,N}^2-\hat\gamma\big((O_n^2)'\big)$.
Combing with the above inequalities, together with (\ref{D11}), we deduce that for any bounded subset $D^+\subset\mathbb{R}^{N}_{+}$,
\begin{equation*}
\begin{aligned}
 W_{p_n}(z)\leq &\left(\frac{\left|O_n^2-O_n^1\right|^2}{\frac{1}{4}\left|O_n^2-O_n^1\right|^2-\left(1+\left|\hat\gamma'\big((O_n^2)'\big)\right|^2\right)(\hat{\epsilon}_{n}^{(2)})^2|z|^2-16L_3^2(\hat{\epsilon}_{n}^{(2)})^2+o_n(1)}\right)^{\frac{2}{p_n-2}(1-\frac{1}{2N})}\\
 \leq & 8^{\frac{N-2}{2}(1-\frac{1}{2N})},~ \quad \forall z\in D^+
\end{aligned}
\end{equation*}
holds for sufficiently large $n$. Hence, $W_{p_n}\to W$ uniformly in $C^{1}$ on compact subsets of $\mathbb{R}^{N}_{+}$. By (\ref{D10}) and the fact that $\hat a_{n,i,j}(z)\to a_{i,j}(0)=\delta_{i,j}$, where $\delta_{i,j}=1$ if $i=j$,  $\delta_{i,j}=0$ if $i\neq j$, we find that $W$ satisfies
\begin{equation}\label{D9}
  \left\{
  \begin{aligned}
  & -\Delta W+\tilde{\lambda}W=\hat \delta W^{\frac{N+2}{N-2}} \quad \mbox{in} \ \mathbb{R}^{N}_{+},\\
  & W=0 \quad \mbox{on} \ \{z=(z',z_N)\in \mathbb{R}^N: z_N=0\},
  \end{aligned}
  \right.
\end{equation}
where $\hat{\delta}$ is defined in (\ref{D18}).

 Next, we show that $m(W)\leq 2.$ By contradiction, we assume that $m(W)\geq 3.$ We take functions $\hat{\eta}_1, \hat{\eta}_2, \hat{\eta}_3\in C^{\infty}_{0}(\mathbb{R}^{N}_{+})$ such that $supp\hat{\eta}_i\subset B_{\hat{R}_i}(\hat{Q}_{i})\subset B_{2\delta}^{+}$ and $\hat{I}_{i,j}:=B_{\hat{R}_i}(\hat{Q}_i)\cap B_{\hat{R}_j}(\hat{Q}_{j})=\emptyset, i\neq j$, where $i,j=1,2,3$ and $\hat{R}_i>0$.
Assume that the following hold
\begin{equation}\label{Y-J-9}
  \int_{\mathbb{R}^{N}_{+}}|\nabla\hat{\eta}_i|^2 dz  + \int_{\mathbb{R}^{N}_{+}}\left(\tilde{\lambda} -(2^*-1)\hat{\delta}W^{\frac{4}{N-2}} \right)\hat{\eta}_i^2dz<0, \ \forall i=1,2,3.
\end{equation}
Set
$$
\hat{\eta}_{i,n}(y):=(\hat{\epsilon}_n^{(2)})^{-\frac{N-2}{2}}\hat{\eta}_{i}\left(\frac{y'-(O_n^2)'}{\hat{\epsilon}_n^{(2)}},\frac{y_N}{\hat{\epsilon}_n^{(2)}}\right), \
~~\forall y=(y', y_N)\in \mathbb{R}^N_+.
$$
Clearly, $supp \hat{\eta}_{i,n}\subset \hat{I}_{i,n}:=B_{\hat{\epsilon}_n^{(2)}\hat{R}_i}\big((O_n^2)'+\hat{\epsilon}_n^{(2)}\hat{Q}_{i}',\hat{\epsilon}_n^{(2)}\hat{Q}_{i,N}\big)$.
Given that $|(O_n^2)'|\leq \delta$, using $\frac{\delta}{\hat{\epsilon}_n^{(2)}}\to+\infty$, it follows that
\begin{eqnarray}\label{0507-1}
B_{\hat{R}_i}(\hat{Q}_i) \subset B_{2\delta}^{+}\subset B_{\frac{2\delta}{\hat{\epsilon}_n^{(2)}}}^{+}\left(-\frac{(O_n^2)'}{\hat{\epsilon}_n^{(2)}},0\right)
\end{eqnarray}
for sufficiently large $n$.

Denote 
$$
z':=\frac{y'-(O_n^2)'}{\hat{\epsilon}_n^{(2)}},~~ z_N:=\frac{y_N}{\hat{\epsilon}_n^{(2)}}.
$$
For any $y\in \hat{I}_{i,n}$, we observe that $(z',z_N)\in B_{\hat{R}_i}(\hat{Q}_i)$.  In combination with (\ref{0507-1}), this implies that $y\in B_{2\delta}^{+}$. Consequently, for sufficiently large $n$, we have $\hat{I}_{i,n}\subset B_{2\delta}^{+}$. In addition,
by $\hat{I}_{i,j}=\emptyset$, we have $\hat{I}_{i,n}\cap \hat{I}_{j,n}=\emptyset, i\neq j$. Hence, $\hat{\eta}_{1,n}, \hat{\eta}_{2,n}, \hat{\eta}_{3,n}$ are orthogonal in $L^2(B_{2\delta}^{+})$.
Furthermore, by (\ref{D12}) and (\ref{Y-J-9}), we have
\begin{align*}
  &\int_{B_{2\delta}^{+}}-\left(\sum_{i,j}a_{i,j}(y)(\hat{\eta}_{\iota,n})_{y_{i}y_{j}}+\sum_{j}b_{j}(y)(\hat{\eta}_{\iota,n})_{y_j}\hat{\eta}_{\iota,n}\right)dy+\int_{B_{2\delta}^{+}}\left(\lambda_n-(p_n-1)w_{p_n}^{p_n-2}\right)\hat{\eta}_{\iota,n}^2dy\nonumber\\
=&\int_{B_{\hat{R}_i}(\hat{Q}_i)}-\left(\sum_{i,j}\hat a_{n,i,j}(z)(\hat{\eta}_{\iota})_{z_{i}z_{j}}+\hat{\epsilon}_n^{(2)}\sum_{j}\hat b_{n,j}(y)(\hat{\eta}_{\iota})_{z_j}\hat{\eta}_{\iota}\right)dz\nonumber\\
&~~~+\int_{B_{\hat{R}_i}(\hat{Q}_i)}\left(\lambda_n(\hat{\epsilon}_n^{(2)})^{2}-(p_n-1)\left(\frac{\hat{\epsilon}_n^{(1)}}{\hat{\epsilon}_n^{(2)}}\right)^{2(1-\frac{1}{N})}W_{p_n}^{p_n-2}\right)\hat{\eta}_{\iota}^2dz\nonumber\\
=&\int_{\mathbb{R}^{N}_{+}}-\left(\sum_{i,j}\hat a_{n,i,j}(z)(\hat{\eta}_{\iota})_{z_{i}z_{j}}+\hat{\epsilon}_n^{(2)}\sum_{j}\hat{b}_{n,j}(y)(\hat{\eta}_{\iota})_{z_j}\hat{\eta}_{\iota}\right)dz\nonumber\\
&~~~+\int_{\mathbb{R}^{N}_{+}}\left(\lambda_n(\hat{\epsilon}_n^{(2)})^{2}-(p_n-1)\left(\frac{\hat{\epsilon}_n^{(1)}}{\hat{\epsilon}_n^{(2)}}\right)^{2(1-\frac{1}{N})}W_{p_n}^{p_n-2}\right)\hat{\eta}_{\iota}^2dz\nonumber\\
\rightarrow&\int_{\mathbb{R}^{N}_{+}}|\nabla\hat{\eta}_{\iota}|^2 dz  + \int_{\mathbb{R}^{N}_{+}}\left(\tilde{\lambda} -(2^*-1)\hat\delta W^{\frac{4}{N-2}} \right)\hat{\eta}_{\iota}^2dz<0, \ \forall \iota=1,2,3.
\end{align*}
This implies that $m(w_{p_n})\geq 3$, which leads to a contradiction.
Thus, we conclude that $m(W)\leq 2$.  By similar arguments as in Lemma \ref{L-criti-blow2}, we can see that (\ref{D9}) has only the trivial solution. However, by the definition of $W_{p_n}$, we have $W_{p_n}(0,\cdots, 0, \frac{l_n}{\hat{\epsilon}_n^{(2)}})=1$. This together with (\ref{D16}) implies that  $W(0,\cdots, 0, l_0)=1$. Therefore, we obtain a contradiction.


Case (ii). $\lambda_n(\hat{\epsilon}_n^{(2)})^2\to+\infty.$
We define
\begin{equation*}
  V_{p_n}(y):=\frac{(\hat{\epsilon}_n^{(1)})^{\frac{-2}{p_n-2}(1-\frac{1}{N})}u_{p_n}(O_n^2+\lambda_n^{-\frac{1}{2}}y)}{A_n(O_n^2)},\quad \forall y\in\tilde{\Omega}_{p_n}:=\lambda_n^{\frac{1}{2}}(\Omega-O_n^2).
\end{equation*}
Then $V_{p_n}$ satisfies
\begin{equation}\label{Y-J-10}
\left\{
\begin{aligned}
  &-\Delta V_{p_n}(y)+V_{p_n}(y)=\frac{u_{p_n}^{p_n-2}(O_n^2)}{\lambda_n}V_{p_n}^{p_n-1}(y), \quad y\in\tilde{\Omega}_{p_n},\\
  &V_{p_n}(0)=1, \quad V_{p_n}(y)=0, \quad y\in\partial\tilde{\Omega}_{p_n}.
\end{aligned}
\right.
\end{equation}
By (\ref{Y-J-0}), we have 
\begin{equation}\label{Y-J-11}
  V_{p_n}(y)=\frac{u_{p_n}(O_n^2+\lambda_n^{-\frac{1}{2}}y)}{u_{p_n}(O_n^2)}\leq\left(\frac{|O_n^2-O_n^1|}{|O_{n}^2+\lambda_n^{-\frac{1}{2}}y-O_n^1|}\right)^{\frac{4}{p_n-2}(1-\frac{1}{2N})}, \ \forall y\in\tilde{\Omega}_{p_n}.
\end{equation}
By (\ref{D11}) and $\lambda_n(\hat{\epsilon}_n^{(2)})^2\to +\infty$, we deduce for any bounded subsets $D\subset\mathbb{R}^{N}$,
\begin{equation}\label{Y-J-12}
  \frac{|O_{n}^2+\lambda_n^{-\frac{1}{2}}y-O_n^1|}{|O_{n}^2-O_n^1|}\to 1
\end{equation}
holds for any $y\in D$. Then $V_{p_n}\to V_0$ in $C^{1}$ uniformly on compact subsets of $\mathbb{R}^{N}$.
Furthermore, by the definition of $A_n(O_n^2)$ and $\hat{\epsilon}_n^{(2)}$, using (\ref{D11}), we get
\begin{equation*}\label{Y-J-4}
\frac{u_{p_n}^{p_n-2}(O_n^2)}{\lambda_n}\to 0.
\end{equation*}
Combining this with (\ref{Y-J-10}), we obtain
\begin{equation*}\label{Y-J-2}
\left\{
\begin{aligned}
  &-\Delta V_0(y)+V_0(y)=0, &\quad& y\in\tilde{\Omega}_0.\\
  &V_0(0)=1, \ V_0=0, &\quad& y\in\partial\tilde{\Omega}_0,
\end{aligned}
\right.
\end{equation*}
where $\tilde{\Omega}_0$ denotes $\mathbb{R}^{N}$ or $\mathbb{R}^{N}_{L_4}:=\{x=(x_1,\cdots,x_N)\in \mathbb{R}^N: x_N>L_4\}$ with $L_4\geq 0$. 
On the other hand, by (\ref{Y-J-11})-(\ref{Y-J-12}), we deduce that $0$ is the maximum point of $V_0,$ which implies that 
\begin{equation*}
-\Delta V_0(0)+V_0(0)>1.
\end{equation*}
This produces a contradiction.
Therefore, (\ref{Y-J-3}) holds.

Now we prove (\ref{Y-J-5}). We observe that  
 $\hat{\Omega}_{p_{n}}^{2}\to \mathbb{R}^N$ or $\hat{\Omega}_{p_{n}}^{2}\to \mathbb{R}^N_{L_5}:=\{x=(x_1,\cdots,x_N)\in \mathbb{R}^N: x_N>L_5\}$ with $L_5\geq 0.$  If the latter case holds, by (\ref{Y-J-3}) and arguments analogous to those in (\ref{D6}), we obtain
\begin{equation*}
\left\{
\begin{aligned}
 &-\Delta \hat{U}_2=\hat{\delta} \hat{U}_2^{2^*-1} &\quad&  \mbox{in} \ \mathbb{R}^N_{L_5},\\
&\hat{U}_2(0)=1,\ \hat{U}_2=0 &\quad&  \mbox{on} \ \partial\mathbb{R}^N_{L_5},
\end{aligned}
\right.
\end{equation*}
where $\hat{\delta}$ is defined in (\ref{D18}).
However, according to Theorem 12 in \cite{F2007}, the above equation admits only the trivial solution. This yields a contradiction.
Therefore, we conclude that (\ref{Y-J-5}) holds.

{\it Step 2. If (\ref{9-30-5}) holds, then (\ref{D22}) holds for $k=2$.}

We assume by contradiction that (\ref{D22}) does not hold for $k=2$. Then there exists $\{O_n^3\}\subset\Omega$ such that
\begin{equation}\label{Y-J-1}
\max\limits_{x\in\Omega}\Theta_n^2(x)=\Theta_n^2(O_n^3)
=\big(d_n^2(O_n^3)\big)^{\frac{2}{p_n-2}(2-\frac{1}{N})}(\hat{\epsilon}_n^{(1)})^{-\frac{2}{p_n-2}(1-\frac{1}{N})}u_{p_n}(O_n^3)\to +\infty.
\end{equation}
Recall the definition of $A_n(x)$, we have 
$$
A_n(O_n^3)=(\hat{\epsilon}_n^{(1)})^{-\frac{2}{p_n-2}(1-\frac{1}{N})}u_{p_n}(O_n^3)\to +\infty.
$$
Denote $$
\hat{\epsilon}_{n}^{(3)}:=A_n(O_n^3)^{-\frac{p_n-2}{2(2-\frac{1}{N})}}.
$$
Using similar arguments as in Step 1, we derive
\begin{equation}\label{D17-Lapla}
\hat{\epsilon}_{n}^{(3)}\ge \hat{\epsilon}_{n}^{(1)}, \quad  \hat{\epsilon}_{n}^{(3)}\to 0,  \quad \frac{|O_n^3-O_n^i|}{\hat{\epsilon}_{n}^{(3)}}\to+\infty \quad \mbox{and} \quad  \frac{|O_n^3-O_n^i|}{\hat{\epsilon}_{n}^{(i)}}\to+\infty, \quad  \forall i\in\{1,2\}.
\end{equation}
Define
\begin{equation*}
  \hat{U}_{p_n,3}(y):=(\hat{\epsilon}_{n}^{(3)})^{\frac{4}{p_{n}-2}(1-\frac{1}{2N})}A_{n}(O_{n}^3+\hat{\epsilon}_{n}^{(3)} y), \quad \forall y\in\hat{\Omega}_{p_{n}}^{3}:=\frac{\Omega-O_{n}^3}{\hat{\epsilon}_{n}^{(3)}}.
\end{equation*}
Then $\hat{U}_{p_n,3}$ satisfies
\begin{equation*}\label{D3-Lapla}
  \left\{
  \begin{aligned}
  & -\Delta \hat{U}_{p_n,3}+\lambda_n(\hat{\epsilon}_n^{(3)})^{2} \hat{U}_{p_n,3}=\left(\frac{\hat{\epsilon}_n^{(1)}}{\hat{\epsilon}_n^{(3)}}\right)^{2(1-\frac{1}{N})}\hat{U}_{p_n,3}^{p_n-1} \quad &\mbox{in}& \  \hat{\Omega}_{p_{n}}^{3},\\
  & \hat{U}_{p_n,3}(0)=1, \quad \hat{U}_{p_n,3}=0 \quad &\mbox{on}& \  \partial\hat{\Omega}_{p_{n}}^{3}.
  \end{aligned}
  \right.
\end{equation*}
Using similar arguments as before, we derive $\lambda_n(\hat{\epsilon}_{n}^{(3)})^{2}\to 0$ and $\hat{U}_{p_n,3}\to \hat{U}_3$ in $C^{1}$ uniformly on compact subsets of $\mathbb{R}^{N}$, where $\hat{U}_3$ satisfies $m(\hat{U}_3)\leq 2$ and
\begin{equation*}\label{D19}
  \left\{
  \begin{aligned}
  & -\Delta \hat{U}_{3}=\tilde{\delta} \hat{U}_{3}^{\frac{N+2}{N-2}} &\quad &\mbox{in} \ \mathbb{R}^{N},\\
  & \hat{U}_{3}(0)=1,
  \end{aligned}
  \right.
\end{equation*}
where 
\begin{equation*}\label{D18'}
\tilde{\delta}:=\lim\limits_{n\to+\infty}\left(\frac{\hat{\epsilon}_n^{(1)}}{\hat{\epsilon}_n^{(3)}}\right)^{2(1-\frac{1}{N})}\in(0,1].
\end{equation*}
Furthermore, there exists $\hat{\phi}_3\in C^{\infty}_{0}(\mathbb{R}^{N})$ such that $supp\hat{\phi}_3\subset B_{\hat{R}_3}$, and $Q_{0,\hat{U}_3}(\hat{\phi}_3)<0$ for some $\hat{R}_3>0$.
Define $$\hat{\phi}_{n,3}(x):=(\hat{\epsilon}_{n}^{(3)})^{-\frac{N-2}{2}}\hat{\phi}_3\left(\frac{x-O_{n}^{3}}{\hat{\epsilon}_{n}^{(3)}}\right),~~ \forall x\in \mathbb{R}^N.
$$ 
It is clear that $supp\hat{\phi}_{n,3}\subset B_{\hat{R}_3\hat{\epsilon}_{n}^{(3)}}(O_n^{3})$. By similar arguments as above, we get for sufficiently large n, $\hat{\phi}_{n,3}\in C^{\infty}_{0}(\Omega)$ and $Q_{\lambda_n,u_{p_{n}}}(\hat{\phi}_{n,3})<0.$
By (\ref{D17-Lapla}), we derive that
 $\{\hat{\phi}_{n,i}\}_{i=1}^{3}$ have disjoint supports and are pairwise orthogonal in $L^{2}(\Omega).$
Hence, we deduce $m(u_{p_n})\geq 3.$ This yields a contradiction. The proof is complete.
\end{proof}

\begin{proof}[Proof of Theorem \ref{above-zero}]
By contradiction, we assume $\lambda_{n}\rightarrow +\infty.$
Consider $\hat{U}_{p_n,1}$ as defined in Lemma  \ref{L-criti-blow2} (ii) and $\hat{U}_{p_n,2}$ as defined in (\ref{D13}). We then have $\hat{U}_{p_n,i}\to \hat{U}_{i}$ in $C^1$ uniformly on compact subsets of $\mathbb{R}^{N}$, where $i\in\{1,k\}$ and $k\in\{1,2\}.$ 
It follows that for any $2<r<2^*$,
\begin{equation}\label{D39}
  \sum_{i=1}^{k}\int_{B_{\hat{R}_i}}|\hat{U}_{p_n,i}|^{r}dx\leq \sum_{i=1}^{k}\hat C^{r}\hat{R}_i^{N},
\end{equation}
where $\hat C:=\max\limits_{i\in\{1,k\}}\{\max\limits_{x\in B_{\hat{R}_i}}\hat{U}_{p_n,i}(x)\}$ and $\hat{R}_i$ is defined in the proof of Theorem \ref{L-criti-blow1}.

Set $m:=2^{*}(1-\frac{3}{8N})$. Clearly, $2<m<2^*$.
In view of $\|u_{p_n}\|^2_2 =c$, we deduce that 
\begin{equation*}\label{D32'}
\int_{\Omega}|u_{p_n}|^{m}dx\geq c^{\frac{m}{2}}|\Omega|^{1-\frac{m}{2}}.
\end{equation*}
Furthermore, since $p_n\to2^*$, for sufficiently large $n$, it is easy to very that $\frac{2m}{p_n-2}-N<0$.
Hence, by (\ref{D39}) and the fact that $\hat{\epsilon}_{n}^{(i)}\to 0,$ we have
\begin{equation}\label{D34}
  (\hat{\epsilon}_n^{(i)})^{\frac{2m}{p_n-2}-N}\int_{\Omega}|u_{p_n}|^{m}dx-\sum_{i=1}^{k}\int_{B_{\hat{R}_i}}|\hat{U}_{p_n,i}|^{m}dx\to+\infty.
\end{equation}

Next, we will divide the proof into two cases. By applying Theorem \ref{L-criti-blow1}, we will show that the left-hand side of (\ref{D34}) is bounded. This will lead to a contradiction. Hence, we conclude that $\{\lambda_n\}$ is bounded from above. 

Case (i). $k=1.$ Denote $r_{n,1}:=\hat{R}_1\hat{\epsilon}_{n}^{(1)}$. By Theorem \ref{L-criti-blow1}, we have
\begin{align*}
  &(\hat{\epsilon}_n^{(1)})^{\frac{2m}{p_n-2}-N}\int_{\Omega}|u_{p_n}|^{m}dx-\int_{B_{\hat{R}_1}}|\hat{U}_{p_n,1}|^{m}dx\\
=&(\hat{\epsilon}_n^{(1)})^{\frac{2m}{p_n-2}-N}\int_{\Omega \setminus B_{r_{n,1}}(O_n^1)} |u_{p_n}|^{m}dxdx\\
\leq&(\hat{\epsilon}_{n}^{(1)})^{\frac{2m}{p_{n}-2}-N}\int_{\Omega\setminus B_{r_{n,1}}(O_{n}^{1})}\Bigg|\frac{C(\hat{\epsilon}_{n}^{(1)})^{\frac{2}{p_{n}-2}(1-\frac{1}{N})}}{|x-O_n^1|^{\frac{2}{p_n-2}(2-\frac{1}{N})}}\Bigg|^{m}dx\\
=&C(\hat{\epsilon}_{n}^{(1)})^{-N}\int_{\Omega\setminus B_{r_{n,1}}(O_{n}^{1})}\Bigg|\frac{x-O_{n}^{1}}{\hat{\epsilon}_{n}^{(1)}}\Bigg|^{-\frac{2m}{p_n-2}(2-\frac{1}{N})}dx\\
=& C\int_{\mathbb{R}^{N}\setminus B_{\hat{R}_1}}|y|^{-\frac{2m}{p_n-2}(2-\frac{1}{N})} dy\\
\leq& C\hat{R}_1^{N-\frac{2m}{p_n-2}(2-\frac{1}{N})}.
\end{align*}
Using $p_n<2^*$ and the choose of $m$, we obtain that for sufficiently large $n$, $N-\frac{2m}{p_n-2}(2-\frac{1}{N})<0$.
Then
\begin{equation*}
  (\hat{\epsilon}_n^{(1)})^{\frac{2m}{p_n-2}-N}\int_{\Omega}|u_{p_n}|^{m}dx-\int_{B_{\hat{R}_1}}|\hat{U}_{p_n,1}|^{m}dx
  \leq C'
\end{equation*}
for some $C'>0.$ This contradicts (\ref{D34}).

Case (ii). $k=2.$  Denote $r_{n,i}:=\hat{R}_i\hat{\epsilon}_{n}^{(i)}, i\in\{1,2\}$. 
 In view of Lemma \ref{L-criti-blow2} (ii) and (\ref{Y-J-5}), we have
\begin{equation*}\label{D35}
  \frac{dist(O_{n}^i,\partial\Omega)}{\hat{\epsilon}_{n}^{(i)}}\rightarrow +\infty.
\end{equation*}
This together with (\ref{D11}), we deduce for sufficiently large $n$,
$$B_{r_{n,i}}(O_{n}^{i})\subset\Omega \quad \mbox{and} \quad B_{r_{n,i}}(O_{n}^{i})\cap  B_{r_{n,j}}(O_{n}^{j})=\emptyset, \quad  i,j\in\{1,2\},\  i\neq j.
$$
For $n$ sufficiently large such that $\frac{2m}{p_n-2}-N<0$, by the fact that $\hat{\epsilon}_n^{(1)}\leq \hat{\epsilon}_n^{(2)}$ and Theorem \ref{L-criti-blow1}, we have
\begin{align*}
&{}(\hat{\epsilon}_{n}^{(2)})^{\frac{2m}{p_n-2}-N}\int_{\Omega}|u_{p_{n}}|^{m}dx-\sum_{i=1}^{2}\int_{B_{\hat{R}_i}}|\hat{U}_{p_n,i}|^{m}dx\\
=&{}(\hat{\epsilon}_{n}^{(2)})^{\frac{2m}{p_n-2}-N}\left(\int_{B_{r_{n,1}}(O_n^1)}|u_{p_{n}}|^{m}dx+\int_{B_{r_{n,2}}(O_n^2)}|u_{p_{n}}|^{m}dx+\int_{\Omega\setminus\bigcup_{i=1}^{2}B_{r_{n,i}}(O_n^i) }|u_{p_{n}}|^{m}dx\right)\\
&{}~~~~~~~~~~~~-\int_{B_{\hat{R}_1}}|\hat{U}_{p_{n},1}|^{m}dx-\int_{B_{\hat{R}_2}}|\hat{U}_{p_{n},2}|^{m}dx\\
\leq&{} (\hat{\epsilon}_{n}^{(1)})^{\frac{2m}{p_n-2}-N}\int_{B_{r_{n,1}}(O_n^1)}|u_{p_{n}}|^{m}dx+(\hat{\epsilon}_{n}^{(2)})^{\frac{2m}{p_n-2}-N}\Bigg(\frac{\hat{\epsilon}_n^{(2)}}{\hat{\epsilon}_n^{(1)}}\Bigg)^{\frac{2m}{p_n-2}(1-\frac{1}{N})}\int_{B_{r_{n,2}}(O_n^2)}|u_{p_{n}}|^{m}dx\\
&{}~~~~~~~~~~~~+(\hat{\epsilon}_{n}^{(2)})^{\frac{2m}{p_n-2}-N}\int_{\Omega\setminus\bigcup_{i=1}^{2}B_{r_{n,i}}(O_n^i) }|u_{p_{n}}|^{m}dx
-(\hat{\epsilon}_n^{(1)})^{\frac{2m}{p_n-2}}\int_{B_{\hat{R}_1}}|u_{p_{n}}(O_{n}^1+\hat{\epsilon}_{n}^{(1)}x)|^{m} dx \\ &{}~~~~~~~~~~~~-(\hat{\epsilon}_n^{(1)})^{-\frac{2m}{p_n-2}(1-\frac{1}{N})}(\hat{\epsilon}_{n}^{(2)})^{\frac{2m}{p_{n}-2}(2-\frac{1}{N})}\int_{B_{\hat{R}_2}}|u_{p_{n}}(O_{n}^2+\hat{\epsilon}_{n}^{(2)}x)|^{m}dx \\
=&{}(\hat{\epsilon}_{n}^{(2)})^{\frac{2m}{p_n-2}-N}\int_{\Omega\setminus\bigcup_{i=1}^{2}B_{r_{n,i}}(O_n^i) }|u_{p_{n}}|^{m}dx\\
\leq&{} (\hat{\epsilon}_{n}^{(2)})^{\frac{2m}{p_n-2}-N}\int_{\Omega\setminus\bigcup_{i=1}^{2}B_{r_{n,i}}(O_n^i) }\Bigg|\frac{C(\hat{\epsilon}_{n}^{(1)})^{\frac{2}{p_n-2}(1-\frac{1}{N})}}{\min\limits_{i\in\{1,2\}}|x-O_n^i|^{\frac{2}{p_n-2}(2-\frac{1}{N})}}\Bigg|^{m}dx\\
\leq &{}C_1(\hat{\epsilon}_{n}^{(2)})^{\frac{2m}{p_n-2}-N}\sum_{i=1}^{2}\int_{\Omega\setminus\bigcup_{i=1}^{2}B_{r_{n,i}}(O_n^i) }\frac{(\hat{\epsilon}_{n}^{(i)})^{\frac{2m}{p_n-2}(1-\frac{1}{N})}}{|x-O_n^i|^{\frac{2m}{p_n-2}(2-\frac{1}{N})}}dx\\
= &{}C_1\Bigg((\hat{\epsilon}_{n}^{(2)})^{\frac{2m}{p_{n}-2}-N}\int_{\Omega\setminus\bigcup_{i=1}^{2}B_{r_{n,i}}(O_n^i) }(\hat{\epsilon}_{n}^{(1)})^{-\frac{2m}{p_n-2}}\bigg|\frac{x-O_n^1}{\hat{\epsilon}_n^{(1)}}\bigg|^{-\frac{2m}{p_n-2}(2-\frac{1}{N})}dx\\
&{}~~~~~~~~~~~~~~~+(\hat{\epsilon}_{n}^{(2)})^{\frac{2m}{p_{n}-2}-N}\int_{\Omega\setminus\bigcup_{i=1}^{2}B_{r_{n,i}}(O_n^i) }(\hat{\epsilon}_{n}^{(2)})^{-\frac{2m}{p_n-2}}\bigg|\frac{x-O_n^2}{\hat{\epsilon}_n^{(2)}}\bigg|^{-\frac{2m}{p_n-2}(2-\frac{1}{N})}dx\Bigg)\\
\leq &{}C_1\Bigg((\hat{\epsilon}_{n}^{(1)})^{\frac{2m}{p_{n}-2}-N}\int_{\Omega\setminus\bigcup_{i=1}^{2}B_{r_{n,i}}(O_n^i) }(\hat{\epsilon}_{n}^{(1)})^{-\frac{2m}{p_n-2}}\bigg|\frac{x-O_n^1}{\hat{\epsilon}_n^{(1)}}\bigg|^{-\frac{2m}{p_n-2}(2-\frac{1}{N})}dx\\
&{}~~~~~~~~~~~~~~~+(\hat{\epsilon}_{n}^{(2)})^{-N}\int_{\Omega\setminus\bigcup_{i=1}^{2}B_{r_{n,i}}(O_n^i) }\bigg|\frac{x-O_n^2}{\hat{\epsilon}_n^{(2)}}\bigg|^{-\frac{2m}{p_n-2}(2-\frac{1}{N})}dx\Bigg)\\
= &{}C_1\Bigg((\hat{\epsilon}_{n}^{(1)})^{-N}\int_{\Omega\setminus\bigcup_{i=1}^{2}B_{r_{n,i}}(O_n^i) }\bigg|\frac{x-O_n^1}{\hat{\epsilon}_n^{(1)}}\bigg|^{-\frac{2m}{p_n-2}(2-\frac{1}{N})}dx+(\hat{\epsilon}_{n}^{(2)})^{-N}\int_{\Omega\setminus\bigcup_{i=1}^{2}B_{r_{n,i}}(O_n^i) }\bigg|\frac{x-O_n^2}{\hat{\epsilon}_n^{(2)}}\bigg|^{-\frac{2m}{p_n-2}(2-\frac{1}{N})}dx\Bigg)\\
\leq  &{}C_1 \left(\int_{\mathbb{R}^{N}\setminus B_{\hat{R}_1}}|y|^{-\frac{2m}{p_n-2}(2-\frac{1}{N})} dy+ \int_{\mathbb{R}^{N}\setminus B_{\hat{R}_2}}|y|^{-\frac{2m}{p_n-2}(2-\frac{1}{N})} dy\right)\\
\leq&{}C_1\left(\hat{R}_1^{N-\frac{2m}{p_n-2}(2-\frac{1}{N})}+\hat{R}_2^{N-\frac{2m}{p_n-2}(2-\frac{1}{N})}\right).
\end{align*}
Since $N-\frac{2m}{p_n-2}(2-\frac{1}{N})<0$ for sufficiently large $n$, it follows that
\begin{equation*}
  (\hat{\epsilon}_{n}^{(2)})^{\frac{2m}{p_{n}-2}-N}\int_{\Omega}|u_{p_{n}}|^{m}dx- \sum_{i=1}^{2}\int_{B_{\hat{R}_{i}}}|\hat{U}_{p_n,i}|^{m}dx\le C''
\end{equation*}
for some $C''>0$. 
Thus, we obtain a contradiction with (\ref{D34}). The proof is complete.
\end{proof}

\section{Proof of Theorem \ref{crit-theo-Lapla}}
In this section, we complete the proof of Theorem \ref{crit-theo-Lapla}. By
(\ref{C6-Lapla}), Lemma \ref{AB-Lapla} and Theorem \ref{above-zero}, it follows that $\{u_{p_{n}}\}$ is bounded in $H_{0}^{1}(\Omega)$. Furthermore, we may assume that $\lambda_{n}\to\hat{\lambda}$ for some $\hat{\lambda}\in\mathbb{R}$.
Using standard arguments as in \cite{W1996}, we can infer the existence of a function $\hat{u}\in \mathcal{S}_{c}$ such that $u_{p_n}\rightharpoonup \hat{u}$ in $H^{1}_{0}(\Omega)$, and $\hat u$ satisfies the equation
\begin{equation}\label{4-2.1}
-\Delta \hat{u}+\hat{\lambda} \hat{u}=\hat{u}^{2^{*}-1}, \quad x\in\Omega
\end{equation}
for some $\hat{\lambda}\in\mathbb{R}.$

In what follows, we show that $u_{p_n}\to \hat{u}$ in $H^{1}_{0}(\Omega)$.
To this end, we extend $u_{p_n}$ and $\hat{u}$ by zero outside $\Omega$ and treat them as elements of $D^{1,2}(\mathbb{R}^{N})$. By Theorem 3.1 and Corollary 3.2 in \cite{T2007} (see also \cite{DS2002,T2013}), we have the following profile decomposition.
\begin{lemma}
There exist sequence $\{U_j\}_{j=1}^{k}\subset D^{1,2}(\mathbb{R}^{N})$, $\{x_{n,j}\}_{j=1}^{k}\subset\overline{\Omega}$ and $\{\sigma_{n,j}\}_{j=1}^{k}\subset(0,+\infty)$ such that
\begin{equation}\label{4-3.7}
  u_{p_n}=\hat{u}+\sum_{j=1}^{k}g_{n}^{(j)}U_{j}+r_n,
\end{equation}
where the transformation $g_n^{(j)}$ is defined by $$g_n^{(j)}u:=g_{\sigma_{n,j},x_{n,j}}u=\sigma_{n,j}^{\frac{N-2}{2}}u(\sigma_{n,j}(\cdot-x_{n,j})), \forall u\in D^{1,2}(\mathbb{R}^{N}).$$
Furthermore, one has that $(g_{n}^{(j)})^{-1}u_{p_n}\rightharpoonup U_{j}$ in $D^{1,2}(\mathbb{R}^{N})$,
$\|r_{n}\|_{L^{2^*}(\mathbb{R}^{N})}\to 0$ and $\sigma_{n,j}\to +\infty$.  
\end{lemma}

Following Lemma 2.5 in \cite{GLW}, we obtain the following result. 
\begin{lemma}\label{Estimate1}
  Suppose that $\{u_{p_n}\}$ has a profile decomposition (\ref{4-3.7}). Then $U_j$ satisfies either the following differential inequality:
  \begin{equation}\label{4-7.6}
    \int_{\mathbb{R}^{N}}\nabla U_{j}\nabla\varphi dx\leq \int_{\mathbb{R}^{N}}U_{j}^{2^*-1}\varphi dx, \quad \forall \varphi\in C^{\infty}_{0}(\mathbb{R}^{N}), \ \varphi\geq 0,
  \end{equation}
  or there exists $L\geq 0$ such that
  \begin{equation}\label{4-7.7}
   \int_{\mathbb{R}^{N}_{L}}\nabla U_{j}\nabla\varphi dx\leq \int_{\mathbb{R}^{N}_{L}}U_{j}^{2^*-1}\varphi dx, \quad \forall \varphi\in C^{\infty}_{0}(\mathbb{R}^{N}_{L}), \ \varphi\geq 0,
  \end{equation}
  where $\mathbb{R}^{N}_{L}:=\{x=(x_{1},\ldots, x_{N})\in \mathbb{R}^N: x_{N}>-L\}$ and $U_{j}=0$ outside of the domain $\mathbb{R}^{N}_{L}$.
\end{lemma}

By applying similar arguments as in Lemma 2.7 of \cite{GLW}, we deduce the following decay estimate for $U_j$.
\begin{lemma}\label{L^2}
Suppose that $U_j$ satisfies the inequality (\ref{4-7.6}) or (\ref{4-7.7}). Then there exists a constant $\tilde C_0$ such that
\begin{equation*}
U_j(x)\leq \frac{\tilde C_0}{(1+|x|^2)^{\frac{N-2}{2}}}.
\end{equation*}
\end{lemma}
Without loss of generality, we may assume that $\sigma_{n,1}=\inf\limits_{j=\{1,\ldots,k\}}\{\sigma_{n,k}\}$. For simplicity, we 
denote $\sigma_{n}=\sigma_{n,1}$ and $x_{n}=x_{n,1}$. Define $D_r(x):=B_r(x)\cap \Omega$.
Set $D_{n}:=D_{4\sigma_n^{-\frac{1}{2}}}(x_n)$.
We define a cut-off function $\phi_n\in C^{\infty}_{0}(\mathbb{R}^{N})$ such that $|\nabla \phi_n|\leq \frac{2}{3}\sigma_n^{\frac{1}{2}}$ and
\begin{equation*}
  \phi_n(x):=
  \left\{
  \begin{aligned}
  &1, &\quad& |x-x_n|\leq \sigma_n^{-\frac{1}{2}},\\
  &0, &\quad& |x-x_n|\geq 4\sigma_n^{-\frac{1}{2}}.
  \end{aligned}
  \right.
\end{equation*}

\begin{proof}[Completion of the proof of Theorem \ref{crit-theo-Lapla}]
It suffices to prove
\begin{equation}\label{S-1-Lapla}
u_{p_n} \to \hat{u} \quad \mbox{in} \ L^{2^*}(\Omega).
\end{equation}
By contradiction, we assume that there exists at least one $U_j\neq 0$ in (\ref{4-3.7}). Without loss of generality, we assume $U_1\neq 0$.
We now consider the following three cases:

Case (i). $x_n\to \hat{x}\in\Omega$. Since $\sigma_n\to+\infty$, it follows that $\sigma_n dist(x_n,\partial\Omega)\to+\infty$, and thus $\Omega_n:=\sigma_n(\Omega-x_n)\to \mathbb{R}^N.$
By Lemma \ref{L^2},  we have $U_1\in L^{2^*}(\mathbb{R}^N)$.
Given that $U_1\neq 0$, we can choose $L_1>0$ large enough such that 
\begin{equation*}\label{positive}
\int_{B_{L_1}}U_1^{2^*} dx >C_1>0
\end{equation*}
for some $C_1>0.$  Again, by Lemma \ref{L^2}, we have $U_1\leq \tilde C_0$. Taking $\frac{N}{N-2}<p<2^*$, we obtain
\begin{equation}\label{positive}
C_1<\int_{B_{L_1}}U_1^{2^*} dx\leq \tilde C_0^{2^*-p}\int_{B_{L_1}}U_1^{p} dx.
\end{equation}
Next, we show that for sufficiently large $n\in\mathbb{N}$,
\begin{equation}\label{S-2-Lapla}
  \int_{D_n} u_{p_n}^p \phi_n dx\geq  C_2\sigma_n^{\frac{p(N-2)-2N}{2}}
\end{equation}
for some $C_2>0$.
Indeed, by setting $D_n^{'}=D_{L_1\sigma_n^{-1}}(x_n)$, for sufficiently large $n$, we have $D_n^{'}\subset D_{\sigma_n^{-\frac{1}{2}}}(x_n)\subset D_n$.
By (\ref{4-3.7}) and the definition of $\phi_n$, we get
\begin{align*}
  \int_{D_n} u_{p_n}^p \phi_n dx
  &\geq  \int_{D_n^{'}} u_{p_n}^p dx\\
  &\geq \int_{D_n^{'}} \hat{u}^p dx+\int_{D_n^{'}}\left(g_{n}^{(1)}U_1\right)^p dx+\int_{D_n^{'}} \Big(\sum_{j\geq 2}^{k}g_{n}^{(j)}U_j\Big)^p dx+\int_{D_n^{'}} r_{n}^p dx\\
  &~~~+\sum_{i=0}^{p}\int_{D_n^{'}} \Big(\hat{u}+\sum_{j=1}^{k}g_{n}^{(j)}U_{j}\Big)^{p-i}r_n^i dx\\
  &\geq \int_{D_n^{'}}\Big(g_{n}^{(1)}U_1\Big)^p dx-\int_{D_n^{'}} \hat{u}^p dx-\int_{D_n^{'}} |r_{n}|^p dx-\sum_{i=0}^{p}\int_{D_n^{'}} \Big|\hat{u}+\sum_{j=1}^{k}g_{n}^{(j)}U_{j}\Big|^{p-i}|r_n|^i dx \nonumber\\
  &:=I_1+I_2+I_3+I_4.
\end{align*}
Here we have used that $\hat{u}$ and $g_{n}^{(j)}U_j$ are nonnegative.

For $I_1$, by the definition of $D_n^{'}$ and $g_{n}^{(1)}U_1$, we have
\begin{align*}
 I_1
= \int_{B_{L_1\sigma_n^{-1}}(x_n)\cap\Omega}\sigma_n^{\frac{N-2}{2}p} U_{1}^p(\sigma_n(x-x_n))dx
=\sigma_n^{\frac{p(N-2)-2N}{2}}\int_{B_{L_1}\cap\Omega_n }U_1^p(y)dy.
\end{align*}
Since $\Omega_n\to\mathbb{R}^{N}$, for large $n$, we have $B_{L_1}\subset\Omega_n$. 
Together with (\ref{positive}), we get
\begin{equation}\label{es-1}
  I_1=\sigma_n^{\frac{p(N-2)-2N}{2}}\int_{B_{L_1}}U_1^p(y)dy\geq \tilde C_0^{p-2^*}C_1\sigma_n^{\frac{p(N-2)-2N}{2}}.
\end{equation}
For $I_2$, since $\{\|u_{p_n}\|_{\infty}\}$ is uniformly bounded and $u_{p_n}\to \hat{u}$ a.e. in $\Omega$, we deduce
\begin{equation}\label{es-2}
I_2\geq -\|\hat{u}\|^p_{\infty}|D_n^{'}|\geq -\|\hat{u}\|^p_{\infty}L_1^{N}\sigma_n^{-N}.
\end{equation}
For $I_3$, by the H\"{o}lder inequality and $\|r_n\|_{L^{2^*}(\mathbb{R}^{N})}\to 0$, it follows that
\begin{align}\label{es-3}
I_3\geq -\left(\int_{D_n^{'}} r_{n}^{2^*} dx\right)^{\frac{p}{2^*}}\left(\int_{D_n^{'}} 1 dx\right)^{\frac{2^*-p}{2^*}}\geq -\left(\int_{D_n^{'}} r_{n}^{2^*} dx\right)^{\frac{p}{2^*}}(L_1^{N}\sigma_n^{-N})^{\frac{2N-p(N-2)}{2N}}=o_n(1)\sigma_n^{\frac{p(N-2)-2N}{2}}.
\end{align}
For $I_4$, by the definition of $g_{n}^{(j)}U_j$, we have 
\begin{align}
I_4
=-\sum_{i=0}^{p}\int_{D_n^{'}}\left|u_{p_n}-r_n\right|^{p-i}|r_n|^i dx
& \geq -\sum_{i=0}^{p}\int_{D_n^{'}} \left(|u_{p_n}|+|r_n|\right)^{p-i}|r_n|^i dx\nonumber\\
&\geq -\sum_{i=0}^{p} 2^{p} \|u_{p_n}\|_{\infty}^{p-i}\left(\int_{D_n^{'}} r_{n}^{2^*} dx\right)^{\frac{i}{2^*}}\left(\int_{D_n^{'}} 1 dx\right)^{\frac{2^*-i}{2^*}}-o_n(1)\sigma_n^{\frac{p(N-2)-2N}{2}}\nonumber\\
&\ge -\sum_{i=0}^{p} 2^{p} \|u_{p_n}\|_{\infty}^{p-i}o_n(1)\sigma_n^{\frac{p(N-2)-2N}{2}}-o_n(1)\sigma_n^{\frac{p(N-2)-2N}{2}}.\label{es-4}
\end{align}
Thus, by (\ref{es-1})-(\ref{es-4}), we obtain \eqref{S-2-Lapla}. 

On the other hand, by the definition of $\phi_n$, we obtain
\begin{equation*}
  \int_{D_n} u_{p_n}^p \phi_n dx\leq \|u_{p_n}\|^p_{\infty}\big|B_{4\sigma_n^{-\frac{1}{2}}}(x_n)\big|\leq \|u_{p_n}\|^p_{\infty}4^N\sigma_n^{-\frac{N}{2}}.
\end{equation*}
This contradicts (\ref{S-2-Lapla}) for sufficiently large $n$.
Therefore, we conclude that $u_{p_n}=\hat{u}+r_n$ and (\ref{S-1-Lapla}) holds.

Case (ii). $x_n\to \hat{x}\in\partial\Omega$ and $\sigma_n dist(x_n,\partial\Omega)\to+\infty$. In this case, we have $\Omega_n\to \mathbb{R}^N.$ The arguments follow similarly to Case (i).

Case (iii). $x_n\to \hat{x}\in\partial\Omega$ and $\sigma_n dist(x_n,\partial\Omega)\to d\geq 0$. Clearly, $\Omega_n\to \mathbb{R}^{N}_{d}$, where $\mathbb{R}^{N}_{d}$ is defined in Lemma \ref{Estimate1}. 
As before,  we denote $D_{r,d}(x):=B_{r,d}(x)\cap \Omega$, where $B_{r,d}(x):=B_r(x)\cap \mathbb{R}^{N}_{d}$.
Set $D_{n,d}:=D_{4\sigma_n^{-\frac{1}{2}},d}(x_n)$.
We define a cut-off function $\tilde{\phi}_n\in C^{\infty}_{0}(\mathbb{R}^{N}_{d})$ such that $|\nabla \tilde{\phi}_n|\leq \frac{2}{3}\sigma_n^{\frac{1}{2}}$ and
\begin{equation*}
  \tilde{\phi}_n(x):=
  \left\{
  \begin{aligned}
  &1, &\quad& |x-x_n|\leq \sigma_n^{-\frac{1}{2}},\\
  &0, &\quad& |x-x_n|\geq 4\sigma_n^{-\frac{1}{2}}.
  \end{aligned}
  \right.
\end{equation*}
Using (\ref{4-7.7}) and Lemma \ref{L^2}, we deduce that $U_1\in L^{2^*}(\mathbb{R}^{N}_{d})$.
By arguments similar to those in Case (i), we can prove that for sufficiently large $n\in\mathbb{N}$,
\begin{equation*}\label{es-5}
  \int_{D_{n,d}} u_{p_n}^p \tilde{\phi}_n dx\geq  C_3\sigma_n^{\frac{p(N-2)-2N}{2}}
\end{equation*}
for some $C_3>0$, where $p$ is chosen as in Case (i). On the other hand, we have
\begin{equation*}\label{es-6}
  \int_{D_{n,d}} u_{p_n}^p \tilde{\phi}_n dx\leq \|u_{p_n}\|^p_{\infty}4^N\sigma_n^{-\frac{N}{2}}.
\end{equation*}
Combining these two inequalities, we derive a contradiction since $\sigma_n\to+\infty$, which implies that (\ref{S-1-Lapla}) holds. Together with (\ref{C1-Lapla}) and (\ref{4-2.1}), we get
\begin{equation}\label{D23}
  u_{p_n}\to \hat{u} \quad  \mbox{in} \quad  H^{1}_{0}(\Omega).
\end{equation}
Furthermore, by $u_{p_n}>0$ and the strong maximum principle, we deduce $\hat{u}>0.$


Finally, we prove that $\hat{u}\neq u_1^*$. By (\ref{D23}) and the definition of $c_{1,p_n}$, for any given $n\in \mathbb{Z}^+$ and sequence $\{\epsilon_n\}$ with $\epsilon_n\to0$, there exists $\gamma_{p_n}\in\Gamma$ such that
\begin{equation}\label{S-5-Lapla} 
   J(\hat{u})+o_n(1)=J_{p_n}(u_{p_n})=c_{1,p_n}\geq J_{p_n}(\gamma_{p_n}(t_n))-\epsilon_n,
\end{equation}
where $t_n$ satisfies $J_{p_n}(\gamma_{p_n}(t_n))=\max\limits_{t\in[0,1]}J_{p_n}(\gamma_{p_n}(t)).$  By Lemma \ref{MP-La-1} and Lemma \ref{AB-Lapla}, we have $c_{1,2^*}\ge 2c\lambda_1(\Omega)$. Then, using (\ref{D30}) and $p_n\to 2^*$, we derive
\begin{align}
  J_{p_n}(\gamma_{p_n}(t_n))-\epsilon_n&=J(\gamma_{p_n}(t_{n,*}))+J_{p_n}(\gamma_{p_n}(t_n))-J_{p_n}(\gamma_{p_n}(t_{n,*}))-J(\gamma_{p_n}(t_{n,*}))+J_{p_n}(\gamma_{p_n}(t_{n,*}))-\epsilon_n\nonumber\\
  &\geq J(\gamma_{p_n}(t_{n,*}))+\frac{1}{2^*}\int_{\Omega}|\gamma_{p_n}(t_{n,*})|^{2^*}dx-\frac{1}{p_n}\int_{\Omega}|\gamma_{p_n}(t_{n,*})|^{p_n}dx-\epsilon_n\nonumber\\
  &\geq c_{1,2^*}+o_n(1)-\epsilon_n> \frac{3}{2}c\lambda_1(\Omega)>c\lambda_1(\Omega)\geq J(u_1^*),\label{S-6-Lapla}
\end{align}
where $t_{n,*}$ satisfies $J(\gamma_{p_n}(t_{n,*}))=\max\limits_{t\in[0,1]}J(\gamma_{p_n}(t))$. By (\ref{S-5-Lapla})-(\ref{S-6-Lapla}), we obtain $J(\hat{u})>J(u_1^*)$, which implies that $\hat{u}\neq u_{1}^{*}$.
Hence, $\hat{u}$ is a second positive normalized solution at the mountain pass level to problem (\ref{1.1-Lapla}) for $\hat{\lambda}\in\mathbb{R}$. This completes the proof.
\end{proof}

\section*{Acknowledgements}
This work is partially supported by the NSFC (12471102), NSF of Jilin Province (20250102004JC),
Research Project of the Education Department of Jilin Province (JJKH20250296KJ), and Beijing
Natural Science Foundation (No. 1242007).

\section*{Data availability}
 Data sharing not applicable to this article as no datasets were generated or analysed during
the current study.

\section*{Declarations}
Conflict of interest The authors declare that they have no conflict of interest.

\end{document}